\documentclass[10pt,reqno]{amsart}
\usepackage{bbm}
\usepackage{cases}
\usepackage{txfonts}
\usepackage{amsfonts}
\usepackage{mathrsfs}
\usepackage{amssymb}
\usepackage{amsmath}
\usepackage{epic}
\newtheorem{proposition}{Proposition}[section]
\newtheorem{lemma}[proposition]{Lemma}
\newtheorem{corollary}[proposition]{Corollary}
\newtheorem{theorem}[proposition]{Theorem}

\theoremstyle{definition}
\newtheorem{definition}[proposition]{Definition}
\newtheorem{example}[proposition]{Example}

\theoremstyle{remark}
\newtheorem{remark}[proposition]{Remark}

\newtheorem{notation}[proposition]{Notation}

\newcommand{\thlabel}[1]{\label{th:#1}}
\newcommand{\thref}[1]{Theorem~\ref{th:#1}}
\newcommand{\selabel}[1]{\label{se:#1}}
\newcommand{\seref}[1]{Section~\ref{se:#1}}
\newcommand{\lelabel}[1]{\label{le:#1}}
\newcommand{\leref}[1]{Lemma~\ref{le:#1}}
\newcommand{\prlabel}[1]{\label{pr:#1}}
\newcommand{\prref}[1]{Proposition~\ref{pr:#1}}
\newcommand{\colabel}[1]{\label{co:#1}}
\newcommand{\coref}[1]{Corollary~\ref{co:#1}}
\newcommand{\relabel}[1]{\label{re:#1}}

\newcommand{\delabel}[1]{\label{de:#1}}

\newcommand{\eqlabel}[1]{\label{eq:#1}}
\newcommand{\equref}[1]{(\ref{eq:#1})}
\newcommand{\nolabel}[1]{\label{no:#1}}
\newcommand{\noref}[1]{Notation~\ref{no:#1}}
\def\a{\alpha}
\def\Ad{\text{Ad}}
\def\b{\beta}

\def\d{\delta}
\def\D{\Delta}

\def\ep{\varepsilon}

\def\g{\gamma}

\def\l{\lambda}

\def\N{\mathbb{N}}
\def\op{\oplus}
\def\ot{\otimes}

\def\ra{\rightarrow}

\def\ti{\times}

\def\<{\leq}
\def\>{\geq}

\date{}
\begin{document}
\title{Generalized Hopf-Ore extensions}
\thanks{This research is supported by NNSF of China (Grant No.11571298)}
\author{Lan You}
\address{College of Mathematical Science, Yangzhou University, Yangzhou 225002, China;
School of Mathematics and Physics, Yancheng Institute
of Technology, Yancheng 224051, China}
\email{youl@ycit.cn}
\author{Zhen Wang}
\address{School of Mathematics and Physics, Yancheng Institute
of Technology, Yancheng 224051, China}
\email{wangz@ycit.cn}
\author{Hui-Xiang Chen}
\address{College of Mathematical Science, Yangzhou University, Yangzhou 225002, China}
\email{hxchen@yzu.edu.cn}
\subjclass[2010]{16T05,16S36,16S30}
\keywords{Hopf algebra, Ore extension, enveloping algebra, half quantum group}

\begin{abstract}
We derive necessary and sufficient conditions for an Ore extension of a Hopf algebra
to have a Hopf algebra structure of a certain type. This construction generalizes
the notion of Hopf-Ore extension, called a generalized Hopf-Ore extension.
We describe the generalized Hopf-Ore extensions of the enveloping algebras of Lie algebras.
For some Lie algebras $\mathfrak g$, the generalized Hopf-Ore extensions of $U(\mathfrak g)$
are classified.
\end{abstract}
\maketitle

\section*{Introduction}
For some special algebras, to investigate Hopf algebra
structures over them is an effective method
for studying and classifying Hopf algebras.
The algebraic structures of Ore extensions have been
studied extensively in the past several years. In particular, one can consider Hopf algebra structures over them.
For example, many Hopf algebras have been constructed and classified by means of
Ore extensions \cite{BDG1,BDG2,Pa,WYC}.
In \cite{Pa}, a class of Hopf algebra structures over Ore extension were defined and studied.
Let $A$ be a Hopf algebra and $A[z;\tau,\d]$ an Ore extension of $A$.
Under certain conditions, $A[z;\tau,\d]$ becomes a Hopf algebra by setting $\D z=z\ot r_1+r_2\ot z$
for some group-like elements $r_1,r_2\in A$, which is called Hopf-Ore
extension of $A$ \cite[Definition 1.0]{Pa}. Recently, general Hopf algebra structures over Ore
extension were discussed in \cite{BOZZ}.

In this paper, we generalize the notion in \cite{Pa} by setting $\D z=z\ot r_1+r_2\ot z+x\ot y$
for some $r_1,r_2,x,y\in A$ such that $A[z;\tau,\d]$ is a Hopf algebra.
We call $A[z;\tau,\d]$ with this type of Hopf algebra structure a generalized Hopf-Ore extension of $A$.
In \cite{Zh}, two kinds of connected Hopf algebras $A(\l_1,\l_2,\a)$ and $B(\l)$
were constructed and were used to classify connected Hopf algebras of GK-dimension three
over an algebraically closed field of characteristic zero. These Hopf algebras
can be regarded as the generalized Hopf-Ore extensions of the enveloping algebras of $2$-dimensional
Lie algebras. Moreover, the half quantum group
$U_q^{\>0}(\mathfrak{sl}(3))$ (see \cite{Ci}) can be constructed
by the generalized Hopf-Ore extension.

This paper is organized as follows.
In \seref{1}, we give a sufficient and
necessary condition for $A[z;\tau,\d]$ to have
this type of Hopf algebra structures and study the properties of
the generalized Hopf-Ore extension $A[z;\tau,\d]$.
For two such Hopf algebra structures over
Ore extensions, we give a sufficient and necessary condition for them to be isomorphic.
In \seref{2}, we describe the generalized Hopf-Ore extensions of the enveloping algebras $U(\mathfrak{g})$
of Lie algebras $\mathfrak g$. The generalized Hopf-Ore extensions of $U(\mathfrak{g})$
are classified when $\mathfrak{g}$ is a $1$-dimensional Lie algebra over an arbitrary field,
a $2$-dimensional Lie algebra over an arbitrary field, and
an $n$-dimensional abelian Lie algebra over a field of characteristic zero, respectively.

Throughout, we work over a field $k$. Let $k^{\ti}$ denote the set of all non-zero elements of $k$,
which is a multiplicative group.
We refer to \cite{Mo} for basic notions concerning Hopf algebras.

\section{Hopf algebra structures on Ore extensions}\selabel{1}
Let $A$ be a $k$-algebra. Let $\tau $ be an algebra endomorphism of
$A$ and $\d$ a $\tau$-derivation of $A$.
The Ore extension $A[z;\tau,\d]$ of the algebra $A$ is an algebra
generated by the variable $z$ and the algebra
$A$ with the relation
\begin{equation}\eqlabel{d1}
    za=\tau (a)z+\d(a),\ a\in A.
\end{equation}
If $\{a_i\mid i\in I\}$ is a $k$-basis of $A$,
 $\{a_iz^j\mid i\in I,j\in \N \}$ is a $k$-basis of $A[z;\tau,\d]$ (see \cite[2.1]{Mc}).

Furthermore, assume that $A$ has a Hopf algebra structure. Then we can define a Hopf algebra structure
on $A[y; \tau, \d]$, which generalizes the Hopf-Ore extension defined in \cite{Pa}.

\begin{definition}\delabel{1.1}
Let $A$ be a Hopf algebra and $H=A[z; \tau, \d]$ an Ore extension of $A$.
If there is a Hopf algebra structure on $H$ such that $A$ is a Hopf subalgebra of $H$ and
$$\D(z)=z\ot r_1+x\ot y +r_2\ot z$$
for some $r_1, r_2, x, y\in A$, then $H$ is called a generalized Hopf-Ore extension of $A$.
In this case, we also say that $H$ has a Hopf algebra structure determined by $(r_1, r_2, x, y)$.
\end{definition}

Note that $r_1, r_2$ are nonzero. When $x\ot y=0$, one recover the definition of usual Hopf-Ore extension in \cite{Pa}.

Let $A$ be a Hopf algebra. Recall that an element $g\in A$ is a {\it group-like element}
if $\D g=g\ot g$ and $\ep(g)=1$.  Let $G(A)$ denote the group of group-like elements in $A$.
For $g, h\in G(A)$, an element $a\in A$ is a {\it $(g,h)$-primitive element} if $\D(a)=a\ot g+h\ot a$.
Let $P_{g,h}(A)$ denote the set of all $(g,h)$-primitive elements of $A$.
When $g=h=1$, a $(1,1)$-primitive element is simply called a {\it primitive element} of $A$,
and $P_{1,1}(A)$ is simply written as $P(A)$.

\begin{proposition}\prlabel{1.2}
Let $A$ be a Hopf algebra, $H=A[z; \tau, \d]$ and $r_1, r_2, x, y\in A$. If
$H$ has a Hopf algebra structure determined by $(r_1, r_2, x, y)$,
then $r_1, r_2\in G(A)$, and one of the followings is satisfied:

{\rm(a)} $x =0$ or $y=0$;\\
{\rm(b)} $x=\a r_2$ and $y=\b r_1$ for some $\a,\b\in k^{\ti}$;\\
{\rm(c)} $x\in P_{r_3,r_2}(A)$ and $y\in P_{r_1,r_3}(A)$ for some $r_3\in G(A)$.
\end{proposition}
\begin{proof}
Note that $H$ is a free left $A$-module under left multiplication with the basis $\{z^i\mid i\>0\}$.
Consequently, $H\ot H\ot H$ is a free left $A\ot A\ot A$-module with the basis $\{z^i\ot z^j\ot z^s\mid i,j,s\>0\}$.
By comparing
\begin{equation*}
    (\D\ot{\rm id})\D(z)=z\ot r_1\ot r_1+x\ot y\ot r_1+r_2\ot z\ot r_1+\D x\ot y +\D r_2\ot z
\end{equation*}
with
\begin{equation*}
    ({\rm id}\ot\D)\D(z)=z\ot \D r_1+x\ot \D y+r_2\ot z\ot r_1+r_2\ot x\ot y+r_2\ot r_2\ot z,
\end{equation*}
one gets
\begin{eqnarray}
 & \D(r_1)=r_1\ot r_1, \quad \D(r_2)=r_2\ot r_2, \eqlabel{1}\\
 & x\ot y\ot r_1+\D(x)\ot y =x\ot\D(y)+r_2\ot x\ot y.\eqlabel{2}
\end{eqnarray}
Since $r_1\neq 0$ and $r_2\neq 0$, it follows from \equref{1} that $r_1$ and $r_2$ are both group-like elements.
We have $x=0$ or $y=0$ or $x,y\neq 0$. When $x,y\neq 0$ and $x=\a r_2$ for some $\a\in k^{\ti}$,
the equation \equref{2} becomes $\a r_2\ot y\ot r_1+\a r_2\ot r_2\ot y=\a r_2\ot\D(y)+r_2\ot\a r_2\ot y$.
Applying $\ep\ot\ep\ot{\rm id}$ to both sides, one gets that $y=\ep(y)r_1$. Let $\b=\ep(y)$.
Then $x=\a r_2$ and $y=\b r_1$ for some $\a,\b\in k^{\ti}$ in this case.
When $x,y\neq 0$ and $x\neq \a r_2$ for any $\a\in k^{\ti}$, $\{r_2,x\}$ are linearly independent.
Hence we have $\D(x)=r_2\ot x+u$ and $\D(y)=y\ot r_1+w$ for some $u, w\in A\ot A$ with $u\neq 0$.
Then it follows from \equref{2} that $u\ot y=x\ot w$, which implies
$u=x\ot u'$ and $w=w'\ot y$ for some non-zero elements $u',w'\in A$.
Hence $x\ot u'\ot y=x\ot w'\ot y$, and so $u'=w'$.
Thus, by $\D(x)=r_2\ot x+x\ot u'$ and $(\D\ot{\rm id})\D(x)=({\rm id}\ot\D)\D(x)$, we have that $\D u'=u'\ot u'$.
Hence $u'$ is a group-like element. This completes the proof by letting $r_3=u'$.
\end{proof}

If $x=0$ or $y=0$, then $x\ot y=0$, and so all Hopf algebra structures on $H$ determined by $(r_1,r_2,x,y)$
are the same. Hence we may assume that $x=y=0$ in this case.
If $x=\a r_2$ and $y=\b r_1$ for some $\a,\b\in k^{\ti}$, let $z'=z+\a\b r_2$, $\d'(a)=\d(a)+\a\b(r_2a-\tau(a)r_2)$ for $a\in A$.
Then $\d'$ is a $\tau$-derivation of $A$, $A[z';\tau,\d']=A[z;\tau,\d]$ as algebras, and $\D z'=z'\ot r_1+r_2\ot z'$.
Thus, $A[z;\tau,\d]$
has a Hopf algebra structure determined by $(r_1,r_2,x,y)$ if and only if $A[z';\tau,\d']$
has a Hopf algebra structure determined by $(r_1,r_2,0,0)$.
In this case, we have $A[z';\tau,\d']= A[z;\tau,\d]$ as Hopf algebras.

\begin{notation}\nolabel{1.3}
Suppose that $H=A[z;\tau,\d]$ has a Hopf algebra structure determined by $(r_1,r_2,x,y)$.
By \prref{1.2} and the discussion above,
we assume in what follows that $x$ is an $(r_3,r_2)$-primitive element,
$y$ is an $(r_1,r_3)$-primitive element for some group-like element $r_3\in A$.
In particular, we assume that $x=0$ if and only if $y=0$.
\end{notation}

\begin{corollary}
Let $A$ be a Hopf algebra and $H=A[z;\tau,\d]$, $r_1,r_2,x,y\in A$. If
$H$ has a Hopf algebra structure determined by $(r_1,r_2,x,y)$,
then
\begin{eqnarray}
&&\ep(z) = 0 \eqlabel{co1},\\
 && S(z)= r_2^{-1}xr_3^{-1}yr_1^{-1}-r_2^{-1}zr_1^{-1}.\eqlabel{an1}
\end{eqnarray}
\end{corollary}
\begin{proof}
The first equation follows from \noref{1.3} and $(\ep\ot{\rm id})\D(z)=z$.
The second equation follows from \noref{1.3} and $(S\ot{\rm id})\D(z)=\ep (z)=0$.
\end{proof}

By \noref{1.3}, we have $\D(z)=z\ot r_1+x\ot y+r_2\ot z$,
$\D(x)=x\ot r_3+r_2\ot x$ and $\D(y)=y\ot r_1+r_3\ot y$. Replacing the
generating element $z$ by $z'=r_3^{-1}z$, the quaternion $(r_1,r_2,x,y)$
by $(r_1',r_2',x',y')=(r_3^{-1}r_1,r_3^{-1}r_2,r_3^{-1}x,r_3^{-1}y)$, we have
$\D(z')=z'\ot r_1'+x'\ot y'+r_2'\ot z'$, $\D(x')=x'\ot 1+r_2'\ot x'$ and
$\D(y')=y'\ot r_1'+1\ot y'$.

\begin{notation}\nolabel{1.5}
Preserving the above notation, we assume in what follows that
if $H=A[z;\tau, \d]$ has a Hopf algebra structure determined by $(r_1,r_2,x,y)$,
then $r_1, r_2\in G(A)$, $x\in P_{1,r_2}(A)$, $y\in P_{r_1, 1}(A)$ and the element $z$
satisfies the relation
\begin{equation}\eqlabel{n1}
    \D(z)=z\ot r_1+x\ot y+r_2\ot z.
\end{equation}
Under this assumption, the equation \equref{an1} becomes
\begin{equation}\eqlabel{n2}
    S(z)=r_2^{-1}(xy-z)r_1^{-1}.
\end{equation}
\end{notation}

Let $Ad_a(b):=\sum a_1bS(a_2)$ for any $a, b\in A$. The next theorem gives
a sufficient and necessary condition for an Ore extension of a Hopf algebra
to have a generalized Hopf-Ore extension structure, which generalizes \cite[Theorem 1.3]{Pa}.

\begin{theorem}\thlabel{1.6}
Let $A$ be a Hopf algebra, $r_1, r_2\in G(A)$, $x\in P_{1,r_2}(A)$, $y\in P_{r_1, 1}(A)$ and $H=A[z;\tau,\d]$. Then
$H$ has a Hopf algebra structure determined by $(r_1,r_2,x,y)$
if and only if the following conditions are satisfied:

{\rm(a)} there is a character $\chi: A\ra k$ such that
$\tau(a)=\sum\chi(a_1)\Ad_{r_1}(a_2)$, $\forall a\in A$;\\
{\rm(b)} $\sum\chi(a_1)\Ad_{r_1}(a_2)=\sum\Ad_{r_2}(a_1)\chi(a_2)$, $\forall a\in A$;\\
{\rm(c)} $\D(\tau(a))(x\ot y)+\D(\d(a))=(x\ot y)\D(a)+\sum\d(a_1)\ot r_1a_2+\sum r_2a_1\ot \d(a_2)$,
$\forall a\in A$.
\end{theorem}

\begin{proof}
Similarly to \cite[Theorem 1.3]{Pa}, the proof can be divided into three steps.
First we show that the comultiplication $\D$ of $A$ can be extended
to $H=A[z;\tau,\d]$ by \equref{n1} if and only if the conditions (a)-(c) are satisfied.
Secondly, we prove that if the conditions (a)-(c) are satisfied,
then $H$ admits an extension of counit from $A$ by \equref{co1}.
Lastly, we show that if the conditions (a)-(c) are satisfied,
then $H$ has antipode $S$ extending the antipode of $A$ by \equref{n2}.

Step 1. Assume that the comultiplication $\D$ of $A$ can be extended to $H$
by \equref{n1}. Then by \equref{d1}, we have $\D(z)\D(a)=\D(\tau(a))\D(z)+\D(\d(a))$.
By \equref{n1}, we have
\begin{equation*}
\begin{split}
    \D(z)\D(a)&=(z\ot r_1+x\ot y+r_2 \ot z)(\sum a_1\ot a_2)\\
    &=\sum\tau(a_1)z\ot r_1a_2+\sum\d(a_1)\ot r_1a_2+\sum(x\ot y)(a_1\ot a_2)\\
    &\quad +\sum r_2a_1\ot\tau(a_2)z+\sum r_2a_1\ot\d(a_2)\\
   &=\sum(\tau(a_1)\ot r_1a_2)(z\ot 1)+\sum(r_2a_1\ot \tau(a_2))(1\ot z)\\
   &\quad +\sum\d(a_1)\ot r_1a_2+\sum(x\ot y)(a_1\ot a_2)+\sum r_2a_1\ot \d(a_2)
\end{split}
\end{equation*}
and
\begin{equation*}
\begin{split}
    \D(\tau(a))\D(z)+\D(\d(a))
    &=\sum(\tau(a)_1\ot \tau(a)_2)(z\ot r_1+x\ot y+r_2\ot z)+\D(\d(a))\\
    &=\sum(\tau(a)_1\ot \tau(a)_2 r_1)(z\ot 1)+\sum(\tau(a)_1 r_2\ot \tau(a)_2)(1\ot z)\\
    &\quad +\sum(\tau(a)_1\ot\tau(a)_2)(x\ot y)+\D(\d(a)).
\end{split}
\end{equation*}
It follows that
\begin{eqnarray}
  \D(\tau(a))&=&\tau(a_1)\ot r_1a_2r_1^{-1} \eqlabel{e1}\\
  \D(\tau(a))&=& r_2a_1r_2^{-1}\ot \tau(a_2)\eqlabel{e2}\\
\nonumber\D(\tau(a))(x\ot y)+\D(\d(a))&=&\sum\d(a_1)\ot r_1a_2+\sum(x\ot y)(a_1\ot a_2)+\sum r_2a_1\ot \d(a_2)
\end{eqnarray}
for any $a\in A$. The last equation coincides with the equation in (c).

Define $\chi: A\ra A$ by $\chi(a):=\sum\tau(a_1)r_1S(a_2)r_1^{-1}$, where $S$ is the antipode of $A$. Then
\begin{equation*}\begin{split}
    \D(\chi(a))&=\sum(\tau(a_1)\ot r_1a_2r_1^{-1})(r_1S(a_4)r_1^{-1}\ot r_1S(a_3)r_1^{-1})\\
    &=\sum\tau(a_1)r_1S(a_4)r_1^{-1}\ot r_1a_2S(a_3)r_1^{-1}\\
    &=\sum\tau(a_1)r_1S(a_2)r_1^{-1}\ot 1=\chi(a)\ot 1.
\end{split}\end{equation*}
Hence $\chi(a)\in k$, and so $\chi$ can be regarded as a linear map from $A$ to $k$.
It is easy to check that $\chi(1)=1$ and $\chi(ab)=\chi(a)\chi(b)$ for any $a, b\in A$.
Hence $\chi$ is a character of $A$. One can recover $\tau $ from $\chi$ as follows:
$\sum\chi(a_1)r_1a_2r_1^{-1}=\sum\tau(a_1)r_1S(a_2)r_1^{-1}r_1a_3r_1^{-1}=\tau(a)$.
This proves (a). Then by \equref{e1} and \equref{e2}, we have
\begin{equation*}
  \sum \chi(a_1)r_1a_2r_1^{-1}\ot r_1a_3r_1^{-1}=\sum r_2a_1r_2^{-1}\ot \chi(a_2)r_1a_3r_1^{-1}.
\end{equation*}
Applying ${\rm id}\ot \ep$ to both sides, one gets  $\sum\chi(a_1)r_1a_2r_1^{-1}=\sum r_2a_1r_2^{-1}\chi(a_2)$.
This proves (b).

Conversely, assume that conditions (a)-(c) are satisfied. Then
\equref{e1} and \equref{e2} are clearly satisfied too. Thus, by the computation above,
one can see that the comultiplication $\D$ of $A$ can be extended
to $H=A[z;\tau,\d]$ by \equref{n1}.

Step 2. Assume that the conditions (a)-(c) are satisfied. Note that $\ep(x)=\ep(y)=0$.
Applying $\ep\ot \ep$ to the equation in (c), one gets
$\ep(\d(a))=\ep(\d(a))+\ep(\d(a))$, and so $\ep(\d(a))=0$ for all $a\in A$.
Thus, if we put $\ep(z)=0$, then $\ep(z)\ep(a)=\ep(\tau(a))\ep(z)+\ep(\d(a))$ for all $a\in A$.
Therefore, $H$ admits an extension of counit from $A$ by \equref{co1}.

Step 3. Assume that the conditions (a)-(c) are satisfied. In order to show that
$H$ has antipode $S$ extending the antipode of $A$ by \equref{n2}, it is
enough to show that if we put $S(z)=r_2^{-1}(xy-z)r_1^{-1}$ then $S(a)S(z)=S(z)S(\tau(a))+S(\d(a))$
for all $a\in A$.

By $\ep(a)1=\sum a_1S(a_2)$, we have
\begin{equation}\eqlabel{e4a}
    0=\d(\ep(a)1)=\sum\d(a_1S(a_2))=\sum\d(a_1)S(a_2)+\sum\tau(a_1)\d(S(a_2)).
\end{equation}
Applying $m\circ({\rm id}\ot S)$ to the equation in (c), we have
\begin{equation*}
    \ep(a)xyr_1^{-1}=\sum\tau(a)_1xyr_1^{-1}S(\tau(a)_2)+\sum\d(a_1)S(a_2)r_1^{-1}+\sum r_2a_1S(\d(a_2)).
\end{equation*}
Then by (a), (b), \equref{e2}, \equref{e4a} and the above equation, we have
\begin{equation*}
    \begin{split}
   S(a)r_2^{-1}xyr_1^{-1}=&\sum S(a_1)r_2^{-1}\ep(a_2)xyr_1^{-1}\\
   =&\sum S(a_1)r_2^{-1}\tau(a_2)_1xyr_1^{-1}S(\tau(a_2)_2)+\sum S(a_1)r_2^{-1}\d(a_2)S(a_3)r_1^{-1}\\
   &+\sum S(a_1)r_2^{-1}r_2a_2S(\d(a_3))\\
  =&r_2^{-1}xyr_1^{-1}S(\tau(a))+\sum S(a_1)r_2^{-1}\d(a_2)S(a_3)r_1^{-1}+S(\d(a))\\
   =&r_2^{-1}xyr_1^{-1}S(\tau(a))-\sum r_2^{-1}\chi(a_1)\d(S(a_2))r_1^{-1}+S(\d(a)).
   \end{split}
\end{equation*}
Again by (a) and (b), we have
\begin{equation*}
S(a)r_2^{-1}=\sum r_2^{-1}\chi(a_1)r_2S(a_3)r_2^{-1}\chi(S(a_2))=r_2^{-1}\chi(a_1)\tau(S(a_2)).
\end{equation*}
Now putting $S(z)=r_2^{-1}(xy-z)r_1^{-1}$. Then for any $a\in A$, we have
\begin{equation*}
    \begin{split}
    S(a)S(z)&=S(a)r_2^{-1}(xy-z)r_1^{-1}=S(a)r_2^{-1}xyr_1^{-1}-S(a)r_2^{-1}zr_1^{-1}\\
    &=r_2^{-1}xyr_1^{-1}S(\tau(a))-\sum r_2^{-1}\chi(a_1)\d(S(a_2))r_1^{-1}+S(\d(a))\\
    &\hspace{0.3cm}-\sum r_2^{-1}\chi(a_1)\tau(S(a_2))zr_1^{-1}\\
    &=r_2^{-1}xyr_1^{-1}S(\tau(a))-r_2^{-1}z\chi(a_1) S(a_2)r_1^{-1}+S(\d(a))\\
    &=r_2^{-1}xyr_1^{-1}S(\tau(a))-r_2^{-1}z r_1^{-1} S(\tau(a))+S(\d(a))\\
    &=S(z)S(\tau(a))+S(\d(a)).
    \end{split}
\end{equation*}
This completes the proof.
\end{proof}

\begin{corollary}\colabel{1.7}
Let $A$ be a Hopf algebra and $H=A[z;\tau,\d]$. If
$H$ has a Hopf algebra structure determined by some $(r_1,r_2,x,y)$ of elements $A$,
then

{\rm(a)} $\chi$ is (convolution) invertible in $A^*$ with $\chi^{-1}=\chi \circ S$,
       where $\chi$ is the character determined by $\tau$ as in \thref{1.6}(a);\\
{\rm(b)} $\tau$ is an algebra automorphism with
       $\tau^{-1}(a)=\sum\chi^{-1}(a_1)r_1^{-1}a_2r_1=\sum r_2^{-1}a_1r_2\chi^{-1}(a_2)$ for all $a\in A$;\\
{\rm(c)} $r_1r_2=r_2r_1$.
\end{corollary}

\begin{proof}
 (a) is known. (b) follows from a straightforward verification.
By \thref{1.6}(b), we have $\sum\chi(a_1)r_1a_2r_1^{-1}=\sum r_2a_1r_2^{-1}\chi(a_2)$.
Taking $a=r_1$, then  $\chi(r_1)r_1r_1r_1^{-1}=r_2r_1r_2^{-1}\chi(r_1)$.
Since $\chi(r_1)\neq 0$, we have $r_1r_2=r_2r_1$. This shows (c).
\end{proof}
\begin{corollary}
Let $A$ be a Hopf algebra and $H=A[z;\tau,\d]$. Assume that
$H$ has a Hopf algebra structure determined by some $(r_1,r_2,x,y)$ of elements of $A$.

{\rm(a)} If $A$ is cocommutative, then $r_1^{-1}r_2,r_2^{-1}r_1\in Z(A)$, the center of $A$;\\
{\rm(b)} If $A$ is commutative, then $\chi\in Z(A^*)$, the center of the dual algebra $A^*$ of $A$.
\end{corollary}

\begin{proof}
(a) Assume that $A$ is cocommutative. It suffices to show $r_1^{-1}r_2\in Z(A)$
since $r_2^{-1}r_1$ is the inverse of $r_1^{-1}r_2$.
By \thref{1.6}(b), $\sum\chi(a_1)r_1a_2r_1^{-1}=\sum r_2a_1r_2^{-1}\chi(a_2)=\sum r_2a_2r_2^{-1}\chi(a_1)$,
which implies that $\sum\chi(a_1)a_2r_1^{-1}r_2=\sum\chi(a_1)r_1^{-1}r_2a_2$  for any $a\in A$.
Hence $ar_1^{-1}r_2=\sum\chi^{-1}(a_1)\chi(a_2)a_3r_1^{-1}r_2
=\sum\chi^{-1}(a_1)\chi(a_2)r_1^{-1}r_2a_3=r_1^{-1}r_2a$  for any $a\in A$,
and so $r_1^{-1}r_2\in Z(A)$.

(b) Assume $A$ is commutative. Then the equation in \thref{1.6}(b) becomes
$\sum\chi(a_1)a_2=\sum a_1\chi(a_2)$ for all $a\in A$.
Hence for any $f\in A^*$, we have $\sum\chi(a_1)f(a_2)=\sum f(a_1)\chi(a_2)$.
Thus, $\chi\in Z(A^*)$.
\end{proof}

\begin{corollary}\colabel{1.9}
If $x\ot y$ satisfies $\D(\tau(a))(x\ot y)=(x\ot y)\D(a)$ for any $a\in A$,
then the equation in  \thref{1.6}(c) becomes
\begin{equation}\eqlabel{t3'}
 \D(\d(a))=\sum\d(a_1)\ot r_1a_2+\sum r_2a_1\ot\d(a_2).
  \end{equation}
\end{corollary}
\begin{proof}
It is clear.
\end{proof}

\begin{notation}\nolabel{1.10}
Let $A$ be a Hopf algebra and $H=A[z;\tau,\d]$ which
has a Hopf algebra structure determined by some $(r_1,r_2,x,y)$ of elements of $A$.
Denote the generalized Hopf-Ore extension $H=A[z;\tau,\d]$ by $H=A(\chi,r_1,r_2,x,y,\d)$, where
$\chi:A\ra k$ is a character such that $\tau(a)=\sum\chi(a_1)\Ad_{r_1}(a_2)$, $r_1$ and $r_2$ are group-like
elements, $x$ is a $(1,r_2)$-primitive element, $y$ is a $(r_1,1)$-primitive element,
and the equations in \thref{1.6}(b) and (c) are satisfied for $\{\chi, r_1, r_2, x, y, \d\}$.
\end{notation}

Two Hopf-Ore extensions $H=A(\chi,r_1,r_2,x,y,\d)$
and $H'=A'(\chi',r_1',r_2',x',y',\d')$ of Hopf algebras  $A$ and $A'$ are said to be isomorphic
if there is a Hopf algebra isomorphism $\Psi:H\ra H'$ such that $\Psi(A)=A'$.

Let $m',1,\D',\ep',S'$ denote the multiplication, the unit,
the comultiplication, the counit and the antipode of $H'$, respectively.

\begin{proposition}\prlabel{1.11}
Two Hopf-Ore extensions $A(\chi,r_1,r_2,x,y,\d)$
and $A'(\chi',r_1',r_2',x',y',\d')$ are isomorphic
if there exists a scalar $\lambda\in k^{\ti}$, a group-like element
$r\in G(A')$, an element $b\in A'$ and a Hopf algebra isomorphism $\Phi: A\ra A'$
such that

{\rm(a)} $\Phi(r_i)=rr_i'$, $i=1,2$,\\
{\rm(b)} $\D'(b)=b\ot r_1'+r_2'\ot b+\lambda r^{-1}\Phi(x)\ot r^{-1}\Phi(y)-x'\ot y'$,
  hence $\ep'(b)=0$,\\
{\rm(c)} $\chi'\Phi =\chi$,\\
{\rm(d)} $\d'=\lambda r^{-1}\Phi\d \Phi^{-1}+\d''$,

where $\d''$ is an inner $\tau'$-derivation of $A'$ defined by $\d''(a')=\tau'(a')b-b a'$
for all $a'\in A'$, $\tau'$ is determined by $\chi'$ as in \thref{1.6}(a).
The converse holds if $A$ (or $A'$) has no zero-divisors.
\end{proposition}
\begin{proof}
Let $H=A(\chi,r_1,r_2,x,y,\d)$ and $H'=A'(\chi',r_1',r_2',x',y',\d')$,
and let $\tau$ and $\tau'$ be the algebra automorphisms of $A$ and $A'$ induced by $\chi$ and $\chi'$
as in \thref{1.6}(a), respectively. Then $\tau$ and $\tau'$ are algebra automorphisms by \coref{1.7}(b).

Assume that there exists some $\lambda\in k^{\ti}$, a group-like element
$r\in G(A')$, $b\in A'$ and a Hopf algebra isomorphism $\Phi: A\ra A'$
such that the conditions (a)-(d) are satisfied.
Let $\Psi(a)=\Phi(a)$ for all $a\in A$ and $\Psi(z)=\lambda^{-1}r(z'+b)$. Then a straightforward
computation shows that
$\Psi$ can be uniquely extended to an algebra isomorphism from $H$ to $H'$. Furthermore, one can check that
$\D'(\Psi(z))=(\Psi\ot \Psi)(\D(z))$ and $\ep'(\Psi(z))=\ep(z)$. Hence $\Psi$ is a
bialgebra isomorphism. Consequently, $\Psi$ is a
Hopf algebra isomorphism since any bialgebra map between two Hopf algebras is a Hopf algebra map.

Conversely, assume that $A$ (or $A'$) has no zero-divisors and that
there is a Hopf algebra isomorphism $\Psi:H\ra H'$ such that $\Psi(A)=A'$.
Then $\Phi=\Psi|_A$ is a Hopf algebra isomorphism from $A$ to $A'$, and so
neither of $A'$ and $A$ has zero-divisors. Hence
$\Psi(z)=\gamma'z'+\eta'$ for some $\gamma', \eta'\in A'$ and $\gamma'\neq 0$.
Similarly, $\Psi^{-1}(z')=\gamma z+\eta$ for some $\gamma, \eta\in A$ and $\gamma\neq 0$.
Thus $z'=\Psi\Psi^{-1}(z')=\Phi(\gamma)\gamma'z'+\Phi(\gamma)\eta'+\Phi(\eta)$.
By comparing the coefficients of $z'$, we have $\Phi(\gamma)\gamma'=1$.
Similarly, one gets $\Phi^{-1}(\gamma')\gamma=1$ from $z=\Psi^{-1}\Psi(z)$.
Applying $\Phi$ on it, we have $\gamma'\Phi(\gamma)=1$.
Thus $\gamma'$ is invertible with the inverse $\Phi(\gamma)$.

Since $\Psi$ is a coalgebra map, $\D'(\Psi(z))=(\Psi\ot \Psi)(\D(z))$
and $\ep'(\Psi(z))=\ep(z)$.
Hence we have
$\D'(\gamma')(z'\ot r_1'+x'\ot y'+r_2'\ot z')+\D'(\eta')
=(\gamma'z'+\eta')\ot\Phi(r_1)+\Phi(x)\ot \Phi(y)+\Phi(r_2)\ot(\gamma'z'+\eta')$
and $\ep'(\eta')=0$. By comparing the two sides of this equation, we have
\begin{eqnarray}
 \D'( \gamma')(1\ot r_1')&=&\gamma'\ot \Phi(r_1), \eqlabel{e4} \\
 \D'(\gamma')(r_2'\ot 1) &=& \Phi(r_2)\ot \gamma', \eqlabel{e5}\\
 \D'(\gamma')(x'\ot y')+\D'(\eta') &=& \eta'\ot\Phi(r_1)+\Phi(x)\ot\Phi(y)+\Phi(r_2)\ot\eta'. \eqlabel{e6}
 \end{eqnarray}
Applying $\ep'\ot{\rm id}$ to both sides of \equref{e4}, we obtain
$\Phi(r_1)=\ep'(\gamma')^{-1}\gamma'r_1'$ since $\gamma'$ is invertible.
Hence $\ep'(\gamma')^{-1}\gamma'$ is a group-like element.
Similarly, applying ${\rm id}\ot\ep'$ to both sides of \equref{e5}, we have
$\Phi(r_2)=\ep'(\gamma')^{-1}\gamma'r_2'$.
Let $\lambda:=\ep'(\gamma')^{-1}$ and $r:=\lambda\gamma'$.
Then $\lambda\neq 0$, $r$ is a group-like element and $\Phi(r_i)=rr_i'$, $i=1,2$.
Let $b=\gamma'^{-1}\eta'$. Then one gets the first equation in (b)
by multiplying $\D'(\gamma'^{-1})$ on the two sides of \equref{e6} from the left.
From $\ep'(\eta')=0$, one gets $\ep'(b)=0$.

Since $\Psi$ is an algebra map, we have $\Psi(z)\Psi(a)=\Psi(\tau(a))\Psi(z)+\Psi(\d(a))$.
This means that
\begin{equation*}
\gamma'\tau'(\Phi(a))z'+\gamma'\d'(\Phi(a))+\eta'\Phi(a)=\Phi(\tau(a))\gamma'z'+\Phi(\tau(a))\eta'+\Phi(\d(a)).
\end{equation*}
By comparing its two sides, we have
\begin{eqnarray}
 \gamma'\tau'(\Phi(a)) &=&\Phi(\tau(a))\gamma', \eqlabel{e7} \\
 \gamma'\d'(\Phi(a))+\eta'\Phi(a) &=& \Phi(\tau(a))\eta'+\Phi(\d(a)). \eqlabel{e8}
\end{eqnarray}
Since $\Phi$ is a Hopf algebra isomorphism, it follows from \thref{1.6}(a) and  \equref{e7} that
\begin{equation*}
   \begin{split}
   \sum\gamma'\chi'(\Phi(a_1))r_1'\Phi(a_2)(r_1')^{-1}&=\sum\chi(a_1)\Phi(r_1)\Phi(a_2)\Phi(r_1)^{-1}\gamma'\\
   &=\sum\chi(a_1)rr_1'\Phi(a_2)(r_1')^{-1}r^{-1}\gamma'\\
   &=\sum\chi(a_1)\gamma'r_1'\Phi(a_2)(r_1')^{-1}.
   \end{split}
\end{equation*}
Hence $\sum\chi'(\Phi(a_1))\Phi(a_2)=\sum\chi(a_1)\Phi(a_2)$. Applying $\ep'$ on its both sides,
one gets $\chi'(\Phi(a))=\chi(a)$ for all $a\in A$, i.e., $\chi'\Phi=\chi$.
Since $\Phi$ is bijective, we may substitute $a$ by $\Phi^{-1}(a')$  in \equref{e8}, where $a'\in A'$. Then we have
\begin{equation*}
    \gamma'\d'(a')+\eta'a'=\gamma'\tau'(a')\gamma'^{-1}\eta'+\Phi\d\Phi^{-1}(a'),
\end{equation*}
here we use the relation \equref{e7}.
Define $\d'': A'\rightarrow A'$ by $\d''(a')=\tau'(a')b-b a'$ for all $a'\in A'$.
Then $\d'=\lambda r^{-1}\Phi\d \Phi^{-1}+\d''$ by $\gamma'^{-1}\eta'=b$ and $\lambda^{-1}r=\gamma'$.
\end{proof}

\begin{corollary}\colabel{1.12}
Let $A(\chi, r_1,r_2,x,y,\d)$ be a generalized Hopf-Ore extension of a Hopf algebra $A$.
Then as generalized Hopf-Ore extensions, we have

{\rm(a)} $A(\chi, r_1,r_2,x,y,\d)\cong A(\chi,r_1,r_2,\a x,\b y, \a\b \d)$ for all $\a,\b\in k^{\ti}$.\\
{\rm(b)} $A(\chi, r_1,r_2,\a(1-r_2),y,\d)\cong A(\chi,r_1,r_2,0,0, \d+\d')$, where $\a\in k$ and $\d'(a)=\a(\tau(a)y-ya)$
for all $a\in A$.\\
{\rm(c)} $A(\chi, r_1,r_2,x,\b(1-r_1),\d)\cong A(\chi,r_1,r_2,0,0, \d+\d'')$, where $\b\in k$ and $\d''(a)=\b(\tau(a)x-xa)$
for all $a\in A$.\\
{\rm(d)} Assume that $x$ and $y$ are linearly dependent and $x\notin kG(A)$. Then $r_1=r_2=1$.
Furthermore, if char$(k)\neq 2$, then $A(\chi, 1,1,x,\a x,\d)\cong A(\chi,1,1,0,0, \d+\d')$,
where $\a\in k$ and $\d'(a)=\frac{1}{2}\a(\tau(a)x^2-x^2a)$ for all $a\in A$.\\
{\rm(e)} $A(\chi, r_1,r_2,x,y,\d)\cong A'(\chi\Phi^{-1},\Phi(r_1),\Phi(r_2),\Phi(x),\Phi(y), \Phi\d\Phi^{-1})$,
where $\Phi:A\ra A'$ is a Hopf algebra isomorphism.
\end{corollary}

\begin{proof}
(a)  Let $A(\chi,r_1,r_2,x,y,\d)=A[z;\tau,\d]$ and $A(\chi,r_1,r_2,\a x,\b y,\a\b\d)=A[z';\tau,\a\b \d]$.
An isomorphism $\Psi: A[z;\tau,\d]\rightarrow A[z';\tau,\a\b \d]$ is given by
$\Psi(a)=a$ for all $a\in A$ and $\Psi(z)=\a^{-1} \b^{-1} z'$.

(b) Let $A(\chi, r_1,r_2,\a(1-r_2),y,\d)=A[z;\tau,\d]$ and $A(\chi,r_1,r_2,0,0, \d+\d')=A[z';\tau,\d+\d']$.
An isomorphism $\Psi: A[z;\tau,\d]\rightarrow A[z';\tau,\d+\d']$ is given by $\Psi(a)=a$ for all $a\in A$ and $\Psi(z)=z'+\a y$.

(c) Let $A(\chi, r_1,r_2,x,\b(1-r_1),\d)=A[z;\tau,\d]$ and $A(\chi,r_1,r_2,0,0, \d+\d'')=A[z';\tau,\d+\d'']$.
An isomorphism $\Psi: A[z;\tau,\d]\rightarrow A[z';\tau,\d+\d'']$ is given by $\Psi(a)=a$ for all $a\in A$ and $\Psi(z)=z'+\b x$.

(d) By \noref{1.5}, $x$ is a $(1,r_2)$-primitive element and
$y$ is an $(r_1,1)$-primitive element. Hence $r_1=r_2=1$ since $x$ and $y$ are linearly dependent
and $x\notin kG(A)$.
Assume that char$(k)\neq 2$. Let $A(\chi, 1, 1, x, \a x, \d)=A[z;\tau,\d]$
and $A(\chi, 1, 1, 0, 0, \d+\d')=A[z';\tau,\d+\d']$.
An isomorphism $\Psi: A[z;\tau,\d]\ra A[z';\tau,\d+\d']$ is given by $\Psi(a)=a$ for all $a\in A$ and $\Psi(z)=z'+\frac{1}{2}\a x^2$.

(e) It follows from \prref{1.11} by putting $\lambda=1, r=1, b=0$ there.
\end{proof}

\begin{example}\relabel{1.13}
Let $A=kG$ be the group algebra of a group $G$ over $k$.
If $H=A[z;\tau,\d]$ has a Hopf algebra structure determined by
$(r_1, r_2, x, y)$, then $x=\a(1-r_2)$, $y=\b(1-r_1)$ for some
$\a,\b\in k^{\ti}$. By \coref{1.12}(b) and (c), $H\cong A(\chi,r_1,r_2,0,0,\d')$
as generalized Hopf-Ore extensions.
Hence $H$ is a usual Hopf-Ore extension.
Thus, it follows from \cite[Propositionp 2.2]{Pa} that
every generalized Hopf-Ore extension of $A=kG$ is isomorphic to a usual Hopf-Ore extension
$A(\chi, 1, r, 0,0,\d)$, where $\chi$ is a group character, $r$ is an element of the center of the group $G$,
and $\d$ is given by $\d(g)=\a(g)(1-r)g$, $g\in G$, for some $1$-cocycle $\a\in Z_{\chi}^1(kG)$.
Note that a $1$-cocycle $\a\in Z_{\chi}^1(kG)$ means a linear map $\a:kG\rightarrow k$
satisfying $\a(gh)=\a(g)+\chi(g)\a(h)$ for all $g, h\in G$.
\end{example}

\begin{example}
The half quantum group $ U_q^{\>0}(\mathfrak{sl}(3))$ is the upper triangular Hopf subalgebra of quantum group
$ U_q(\mathfrak{sl}(3))$. Let $A$ be a Hopf algebra generated by $K_1,K_2,K_1^{-1}$, $K_2^{-1},E_1,E_2$, subject to
the relations
\begin{eqnarray}
&&K_iK_j=K_jK_i,\, K_iK_i^{-1}=K_i^{-1}K_i=1, \\
&&K_iE_jK_i^{-1}=q^{a_{ji}}E_j.
\end{eqnarray}
Then $A$ is a Hopf algebra with the comultiplication $\D$, counit $\ep$ and antipode $S$ given by
$$\begin{array}{l}
\D(E_i)=K_i\ot E_i+E_i\ot 1,\quad \D(K_i)=K_i\ot K_i,\quad \D(K_i^{-1})=K_i^{-1}\ot K_i^{-1},\\
\ep (E_i)=0,\quad  \ep (K_i)=1, \quad S(E_i)=-K_i^{-1}E_i,\quad S(K_i)=K_i^{-1},\\
\end{array}$$
where $q\in k^{\ti}$, $1\<i,j\<2$ and $a_{11}=a_{22}=2$, $a_{12}=a_{21}=-1$.
Define $\tau(K_i)=q^{-1}K_i$, $\tau(K_i^{-1})=qK_i^{-1}$ for $i=1,2$,
$\tau(E_1)=q^{-1}E_1$ and $\tau(E_2)=qE_2$. Then $\tau$ can be uniquely
extended to an algebra automorphism of $A$.
Let $H=A[z;\tau,0]$ be the Ore extension of $A$.
Taking $r_1=1,r_2=K_1K_2,x=(q-q^{-1})K_2E_1$ and $y=E_2$. Then $(q-q^{-1})K_2E_1$
is a $(K_2,K_1K_2)$-primitive element and $E_2$ is a $(1,K_2)$-primitive element.
Let $\D z=z\ot 1+(q-q^{-1})K_2E_1\ot E_2+K_1K_2\ot z$. By \prref{1.2}, $H$ has a
Hopf algebra structure determined by $(1,K_1K_2,(q-q^{-1})K_2E_1,E_2)$.
Let $I=\langle z-E_1E_2+q^{-1}E_2E_1\rangle$ be the ideal of $H$ generated by $z-E_1E_2+q^{-1}E_2E_1$.
Then one can check that $I$ is a Hopf ideal of $H$. Moreover, $H/I$ is isomorphic to the half quantum group
$ U_q^{\>0}(\mathfrak{sl}(3))$ as a Hopf algebra.
\end{example}

\section{Classification of generalized Hopf-Ore extensions of some Hopf algebras}\selabel{2}

In this section, we investigate the generalized Hopf-Ore extensions for
the enveloping algebras of Lie algebras.
In particular, we classify the generalized Hopf-Ore extensions of $U(\mathfrak{g})$
when $\mathfrak{g}$ is a $1$-dimensional or $2$-dimensional Lie algebra
over an arbitrary field, or an $n$-dimensional abelian Lie algebra over a field with characteristic zero.

By the discussion in the last section, each generalized Hopf-Ore extension $A[z;\tau,\d]$
of a Hopf algebra $A$ can be written as $A(\chi,r_1,r_2,x,y,\d)$, where
$\chi:A\ra k$ is a character, $r_1,r_2\in G(A)$, $x\in P_{1,r_2}(A)$, $y\in P_{r_1,1}(A)$
and \thref{1.6}(b)-(c) are satisfied for $(\chi,r_1,r_2,x,y,\d)$.
We also assume that $x=0$ if and only if $y=0$.

\subsection{Generalized Hopf-Ore extensions of $U(\mathfrak g)$}

Let $\mathfrak{g}$ be a Lie algebra and $U(\mathfrak g)$ its enveloping algebra.
Then $U(\mathfrak g)$ is a Hopf algebra as usual. It is well-known that $U(\mathfrak g)$ has no zero-divisors.
Let $\{a_{\l}\mid \l\in \Lambda\}$ be a fixed ordered basis of $\mathfrak g$.
Let us use the notations in \cite[pp. 73]{Mo}. Say functions $\textbf{n}:\Lambda \ra \N$ with finite support if
$\textbf{n}(\l)\neq 0 \Leftrightarrow \l\in \{\l_1,\cdots, \l_m\}$, where $\l_1<\l_2<\cdots<\l_m$.
Then $a^{\textbf{n}}$ denotes the basis monomial
$a_{\l_1}^{\textbf{n}(\l_1)}a_{\l_2}^{\textbf{n}(\l_2)}\cdots a_{\l_m}^{\textbf{n}(\l_m)}$, and any
$b\in U(\mathfrak{g})$ may be written as $b=\sum_{\textbf{n}}\a_{\textbf{n}}a^{\textbf{n}}$, where
$\a_{\textbf{n}}\in k$ are almost all zero. Define $\textbf{m}\<\textbf{n}$ if $\textbf{m}(\l)\<\textbf{n}(\l)$ for
all $\l\in \Lambda$. Let $\textbf{n}!=\Pi_{\l\in \Lambda}\textbf{n}(\l)!$ and $\begin{pmatrix}
\textbf{n} \\
\textbf{m} \\
\end{pmatrix}=\Pi_{\l\in \Lambda}\begin{pmatrix}
\textbf{n}(\l) \\
\textbf{m}(\l) \\
\end{pmatrix}$. Then for all basis monomials $a^{\textbf{n}}$,
\begin{equation}\eqlabel{e9a}
    \D a^{\textbf{n}}=\sum_{0\<\textbf{m}\<\textbf{n}}\begin{pmatrix}
\textbf{n} \\
\textbf{m} \\
\end{pmatrix}a^{\textbf{m}}\ot a^{\textbf{n-m}}.
\end{equation}
\begin{lemma}\lelabel{2.1}(\cite[Proposition 5.5.3]{Mo})
Let $\mathfrak{g}$ be a Lie algebra over $k$ and $H=U(\mathfrak g)$. Then

{\rm(a)} $H$ is connected with $H_0=k1$.\\
{\rm(b)} If char$k=0$, then $P(H)=\mathfrak{g}$.\\
{\rm(c)} If char$k=p>0$, then $P(H)=\hat{\mathfrak{g}}$, the
restricted Lie algebra spanned by all $\{a^{p^r}\mid a\in \mathfrak{g},r\>0\}$.
\end{lemma}

When char$k=p>0$, it follows from \leref{2.1} that $\{a_{\l}^{p^r}\mid \l\in \Lambda, r\>0\}$ is
a basis of $\hat{\mathfrak{g}}$. This is a totally ordered set if we
define $ a_{\l}^{p^r}\<a_{\l'}^{p^s}$ when $\l<\l'$ or $\l=\l'$ and $r\<s$. In particular, $a_{\l}^{p^r}=a_{\l'}^{p^s}$ if and only if $\l=\l'$ and $r=s$.

\begin{lemma}\lelabel{2.2}
Let $I$ be an ordered set and $B=\{b_i\mid i\in I\}$ a basis of $\mathfrak{g}$ (resp., $\hat{\mathfrak{g}}$)
when char$k=0$ (resp., char$k=p>0$). Let
$c\in U(\mathfrak{g})$.

{\rm(a)} If char$k=0$ (resp., char$k=p>2$) and $\D(c)=c\ot 1+1\ot c+\sum_{i,j\in I}\a_{ij}b_i\ot b_j$
for some $\a_{ij}\in k$, then $\a_{ij}=\a_{ji}$ for any $i,j\in I$, i.e.,
$$\D(c)=c\ot 1+1\ot c+\sum_{i,j\in I,i<j}\a_{ij}(b_i\ot b_j+b_j\ot b_i)+\sum_{i\in I}\a_{ii}b_i\ot b_i.$$
In this case, $c-\sum_{i,j\in I,i<j}\a_{ij}b_ib_j-\frac{1}{2}\sum_{i\in I}\a_{ii}b_i^2\in \mathfrak{g}$
(resp., $\in \hat{\mathfrak{g}}$).\\
{\rm(b)} If char$k=2$ and $\D(c)=c\ot 1+1\ot c+\sum_{i,j\in I}\a_{ij}b_i\ot b_j$
for some $\a_{ij}\in k$, then $\a_{ij}=\a_{ji}$ for any $i\neq j$ in $I$ and $\a_{ii}=0$ for any $i\in I$, i.e.,
$$\D(c)=c\ot 1+1\ot c+\sum_{i,j\in I,i<j}\a_{ij}(b_i\ot b_j+b_j\ot b_i).$$
In this case, $c-\sum_{i,j\in I,i<j}\a_{ij}b_ib_j\in \hat{\mathfrak{g}}$.
\end{lemma}
\begin{proof}
(a) follows from the cocommutativity of $U(\mathfrak{g})$ immediately.
For (b), by char$k=2$, $\D(b^2)=b^2\ot 1+1\ot b^2$ for any $b\in \hat{\mathfrak{g}}$.
Hence $b_i\ot b_i$ can't appear in the expression of $\D(c)$.
\end{proof}

The usual Hopf-Ore extension of $U(\mathfrak{g})$ has been
classified in \cite[Proposition 2.3]{Pa} when char$k=0$. We will discuss the generalized Hopf-Ore
extension of $U(\mathfrak{g})$ for an arbitrary field $k$.

Let $H=U(\mathfrak{g})[z;\tau,\d]$ be an Ore extension of $U(\mathfrak{g})$.
If $H$ has a Hopf algebra structure determined by $(r_1,r_2,x,y)$,
then $r_1=r_2=1$, $x$ and $y$ are primitives by \leref{2.1}.
Since $r_1=r_2=1$ and $U(\mathfrak{g})$ is cocommutative,
any character $\chi$ of $U(\mathfrak{g})$ satisfies \thref{1.6}(b).
Since $U(\mathfrak{g})$ is generated by $\mathfrak{g}$ as an algebra,
\thref{1.6}(a) is equivalent to that there exists a  character $\chi$ of $U(\mathfrak{g})$
such that $\tau(a)=a+\chi(a)$ for all $a\in \mathfrak{g}$.
In this case, the equation in \thref{1.6}(c) becomes
\begin{equation}\eqlabel{t3}
\D(\tau(a))(x\ot y)+\D(\d(a))=(x\ot y)\D(a)+\sum\d(a_1)\ot a_2+\sum a_1\ot \d(a_2),
\end{equation}
where $a\in U(\mathfrak{g})$. We claim that if $a, b\in U(\mathfrak{g})$ satisfy \equref{t3}
then so does $ab$. In fact, since $\D(\tau(u))=\D(\sum\chi(u_1)u_2)=\sum\tau(u_1)\ot u_2=\sum u_1\ot\tau(u_2)$
for any $u\in U(\mathfrak{g})$, we have
$$\begin{array}{rl}
&\D(\tau(ab))(x\ot y)+\D(\d(ab))\\
=&\D(\tau(a))\D(\tau(b))(x\ot y)+\D(\d(a))\D(b)+\D(\tau(a))\D(\d(b))\\
=&\D(\tau(a))(\D(\tau(b))(x\ot y)+\D(\d(b)))+\D(\d(a))\D(b)\\
=&\D(\tau(a))((x\ot y)\D(b)+\sum\d(b_1)\ot b_2+\sum b_1\ot \d(b_2))+\D(\d(a))\D(b)\\
=&(\D(\tau(a))(x\ot y)+\D(\d(a)))\D(b)+\D(\tau(a))(\sum\d(b_1)\ot b_2+\sum b_1\ot \d(b_2))\\
=&((x\ot y)\D(a)+\sum\d(a_1)\ot a_2+\sum a_1\ot \d(a_2))\D(b)\\
&+\sum\tau(a_1)\d(b_1)\ot a_2b_2+\sum a_1b_1\ot\tau(a_2)\d(b_2)\\
=&(x\ot y)\D(a)\D(b)+\sum\d(a_1)b_1\ot a_2b_2+\sum a_1b_1\ot \d(a_2)b_2\\
&+\sum\tau(a_1)\d(b_1)\ot a_2b_2+\sum a_1b_1\ot\tau(a_2)\d(b_2)\\
=&(x\ot y)\D(ab)+\sum\d(a_1b_1)\ot a_2b_2+\sum a_1b_1\ot \d(a_2b_2)\\
=&(x\ot y)\D(ab)+\sum\d((ab)_1\ot (ab)_2+\sum(ab)_1\ot\d((ab)_2).\\
\end{array}$$
This shows the claim. Thus, \thref{1.6}(c) is equivalent to that \equref{t3}
is satisfied for any $a\in \mathfrak{g}$ since $U(\mathfrak g)$ is generated
by $\mathfrak g$ as an algebra. On the other hand, when $c\in P(U(\mathfrak g))$,
\equref{t3} becomes
\begin{equation}\eqlabel{e10}
    \D(\d(c))=\d(c)\ot 1+1\ot \d(c)+[x,c]\ot y+x\ot [y,c]-\chi(c)x\ot y,
\end{equation}
where $[u, v]=uv-vu$ for any $u, v\in U(\mathfrak g)$.
Summarizing the discussion above together with \thref{1.6},
we have the following proposition.

\begin{proposition}\prlabel{2.3}
Let $H$ be a generalized Hopf-Ore extension of $U(\mathfrak{g})$.
Then $H$ is isomorphic to $U(\mathfrak g)(\chi,1,1,x,y,\d)$
for some $x,y\in P(U(\mathfrak{g}))$, a character $\chi$ of $U(\mathfrak{g})$
and a $\tau$-derivation $\d$ of $U(\mathfrak{g})$ such that
\equref{e10} is satisfied for all $c\in\mathfrak{g}$,
where $\tau$ is an algebra automorphism of $U(\mathfrak{g})$ determined by
$\tau(c)=c+\chi(c)$, $\forall c\in\mathfrak{g}$.
\end{proposition}

In what follows, denote $U(\mathfrak g)(\chi,1,1,x,y,\d)$ by $U(\mathfrak g)(\chi,x,y,\d)$ simply.
Let ${\rm Aut_{Hopf}}(U(\mathfrak g))$ denote the group of all Hopf algebra automorphisms of $U(\mathfrak g)$.

\subsection{The case of dim($\mathfrak{g}$)=1}
Throughout this subsection, let $\mathfrak{g}=ka$ be a $1$-dimensional Lie algebra.
Then $G(U(\mathfrak{g}))=\{1\}$ by \leref{2.1}(a).

For any $\a\in k$,
there is a character $\chi_{\a}: U(\mathfrak g)\ra k$ determined by $\chi_{\a}(a)=\a$.
Moreover, any character of $U(\mathfrak g)$ is equal to some $\chi_{\a}$ with $\a\in k$.
Note that $\chi_0=\ep$, the counit of $U(\mathfrak g)$.
Each character $\chi_{\a}$ induces an algebra automorphism $\tau_{\a}$
of $U(\mathfrak g)$ given by $\tau_{\a}(a)=a+\a$ as stated in \prref{2.3}.
In particular, $\tau_0$ is exactly the identity map on $U(\mathfrak g)$,
and a $\tau_0$-derivation is a usual derivation. Obviously,
a derivation $\d$ of $U(\mathfrak g)$ is uniquely determined by the value $\d(a)$.
Let $\delta_0$ be the derivation of $U(\mathfrak g)$ determined by $\delta_0(a)=a$.

For any $\a\in k^{\times}$, one can define a Hopf algebra automorphism $\Phi_{\a}$
of $U(\mathfrak g)$ by $\Phi_{\a}(a)=\a a$.
The we have the following lemma.

\begin{lemma}\lelabel{2.5}
The map $\Phi: k^{\times}\ra {\rm Aut}_{\rm Hopf}(U(\mathfrak{g}))$, $\a\mapsto\Phi_{\a}$,
is a group isomorphism.
\end{lemma}
\begin{proof}
It follows from a straightforward verification.
\end{proof}

Now let $H$ be a generalized Hopf-Ore extension of $U(\mathfrak{g})$
Then by \prref{2.3}, $H=U(\mathfrak{g})(\chi_{\a}, x, y, \d)$,
where $\a\in k$, $x, y\in P(U(\mathfrak g))$ and $\d$ is a $\tau_{\a}$-derivation of $U(\mathfrak{g})$
such that \equref{e10} is satisfied for all $c\in\mathfrak{g}$.

In case char$k=0$, $x=\b a$ and $y=\g a$ for some $\b,\g\in k$ by \leref{2.1}.
Note that we always assume that $x=0$ if and only if $y=0$. By \coref{1.12}(d),
$H$ is isomorphic to a usual Hopf-Ore extension $U(\mathfrak{g})(\chi_{\a},0,0,\d')$.
Then by \equref{e10}, $\D(\d'(a))=\d'(a)\ot 1+1\ot\d'(a)$, and so $\d'(a)\in\mathfrak{g}$ by \leref{2.1}.

\begin{proposition}\prlabel{2.5}
Assume char$k=0$. Then up to isomorphism, there are three generalized Hopf-Ore extensions of $U(\mathfrak{g})$:
$U(\mathfrak{g})(\ep,0,0,0)$, $U(\mathfrak{g})(\ep,0,0,\d_0)$,
$U(\mathfrak{g})(\chi_1,0,0,0)$, where $\d_0$ is the derivation of $U(\mathfrak g)$
given before.
\end{proposition}
\begin{proof}
We first show that $U(\mathfrak{g})(\ep,0,0,0)$, $U(\mathfrak{g})(\ep,0,0,\d_0)$ and
$U(\mathfrak{g})(\chi_1,0,0,0)$ are not isomorphic to each other as generalized
Hopf-Ore extensions of $U(\mathfrak{g})$.
Since $\ep\Phi=\ep\neq\chi_1$ for any Hopf algebra automorphism $\Phi$ of $U(\mathfrak g)$,
it follows from \prref{1.11} that both $U(\mathfrak{g})(\ep,0,0,0)$ and $U(\mathfrak{g})(\ep,0,0,\d_0)$
are not isomorphic to $U(\mathfrak{g})(\chi_1,0,0,0)$.
Suppose that $U(\mathfrak{g})(\ep,0,0,0)$ and $U(\mathfrak{g})(\ep,0,0,\d_0)$
were isomorphic generalized Hopf-Ore extensions of $U(\mathfrak g)$.
Since $G(U(\mathfrak g))=\{1\}$, it follows from \prref{1.11} that
there exists an element $b\in U(\mathfrak{g})$
such that $\D(b)=b\ot 1+1\ot b$ and $\d_0=\d''$, where $\d''$ is an inner derivation of
$U(\mathfrak g)$ defined by $\d''(a')=a'b-ba'$ for all $b'\in U(\mathfrak g)$. Since $U(\mathfrak g)$ is commutative,
$\d''=0$. This shows $\d_0=0$, a contradiction. Therefore, $U(\mathfrak{g})(\ep,0,0,0)$ and $U(\mathfrak{g})(\ep,0,0,\d_0)$
are not isomorphic generalized Hopf-Ore extensions of $U(\mathfrak g)$.

Now let $H$ be a generalized Hopf-Ore extensions on $U(\mathfrak g)$.
By the discussion before, we may assume that $H=U(\mathfrak{g})(\chi_{\a},0,0,\d)$,
where $\a\in k$ and $\d$ is a $\tau_{\a}$-derivation of $U(\mathfrak g)$ with $\d(a)\in\mathfrak{g}$.
Hence $\d(a)=\eta a$ for some $\eta\in k$.
If $\a=0$ and $\eta=0$, then $\chi_{\a}=\chi_0=\ep$ and $\d=0$, and hence $H=U(\mathfrak{g})(\ep,0,0,0)$.
If $\a=0$ and $\eta\neq 0$, then $\d=\eta\d_0$, and so it follows from \coref{1.12}(a) that
$H$ is isomorphic to $U(\mathfrak{g})(\ep,0,0,\d_0)$ as a generalized Hopf-Ore extension of $U(\mathfrak g)$.
Finally, if $\a\neq 0$, then using \prref{1.11} with $\lambda=1$, $r=1$, $b=-\d(a)$ and $\Phi=\Phi_{\a}$,
one can check that $H$ is isomorphic to $U(\mathfrak{g})(\chi_1,0,0,0)$ as a generalized Hopf-Ore extension of $U(\mathfrak g)$.
\end{proof}

In case char$k=p>0$, $x,y\in \hat{\mathfrak{g}}$ by \leref{2.1}.
If $x=y=0$, then $H=U(\mathfrak{g})(\chi_{\a},0,0,\d)$, and $\d(a)\in\hat{\mathfrak g}$
as above.
Now assume that $x\neq 0$ and $y\neq 0$.
Then \equref{e10} becomes
\begin{equation*}
\D(\d(a))=\d(a)\ot 1+1\ot \d(a)-\a x\ot y.
\end{equation*}

If char$k=p=2$, then it follows from \leref{2.2}(b) that $\a=0$.
In general, if $\a=0$, then $\tau_0(a)=a$ and $\d(a)\in \hat{\mathfrak{g}}$.
Now assume that $\a\neq 0$. Then char$k=p>2$. By \leref{2.2}(a), we have that $x=\lambda y$
for some $\lambda\in k^{\times}$.
By \coref{1.12}(d), $H$ is isomorphic to a usual Hopf-Ore extension.
Thus, we have the following proposition.

\begin{proposition}\prlabel{2.6}
Assume that char$k=p>0$. Then up to isomorphism, there are four classes of generalized Hopf-Ore extensions of
$U(\mathfrak{g})$ as follows:

{\rm(a)} $U(\mathfrak{g})(\ep,0,0,0)$;\\
{\rm(b)} $U(\mathfrak{g})(\ep,0,0,\d_1)$, where $\d_1$ is a derivation of $U(\mathfrak{g})$
 with $0\neq\d_1(a)\in \hat{\mathfrak{g}}$;\\
{\rm(c)} $U(\mathfrak{g})(\chi_1,0,0,0)$;\\
{\rm(d)} $U(\mathfrak{g})(\ep,x,y,\d_2)$, where $0\neq x,y\in\hat{\mathfrak{g}}$ with $kx\neq ky$,
and $\d_2$ is a derivation of $U(\mathfrak{g})$ with $\d_2(a)\in \hat{\mathfrak{g}}$.

Moreover, $U(\mathfrak{g})(\ep,0,0,0)$, $U(\mathfrak{g})(\ep,0,0,\d_1)$, $U(\mathfrak{g})(\chi_1,0,0,0)$ and
$U(\mathfrak{g})(\ep,x,y,\d_2)$ are pairwise non-isomorphic generalized Hopf-Ore extensions on $U(\mathfrak{g})$.
\end{proposition}
\begin{proof}
Let $H$ be a generalized Hopf-Ore extension of $U(\mathfrak{g})$. By the discussion above,
we may assume $H=U(\mathfrak{g})(\chi_{\a},0,0,\d)$ for some $\a\in k$ and $\tau_{\a}$-derivation $\d$
with $\d(a)\in\hat{\mathfrak g}$,
or $H=U(\mathfrak{g})(\ep,x,y,\d_2)$ for some $0\neq x, y\in\hat{\mathfrak{g}}$ with $kx\neq ky$,
and derivation $\d_2$ with $\d(a)\in\hat{\mathfrak{g}}$.

Similarly to the proof of \prref{2.5}, one can show
that as generalized Hopf-Ore extensions of $U(\mathfrak{g})$,
$U(\mathfrak{g})(\ep,0,0,0)$, $U(\mathfrak{g})(\ep,0,0,\d_1)$ and $U(\mathfrak{g})(\chi_1,0,0,0)$ are pairwise non-isomorphic,
$U(\mathfrak{g})(\chi_1,0,0,0)$ and $U(\mathfrak{g})(\ep,x,y,\d_2)$ are not isomorphic,
and $U(\mathfrak{g})(\chi_{\a},0,0,\d)$ is isomorphic to one of
$U(\mathfrak{g})(\ep,0,0,0)$, $U(\mathfrak{g})(\ep,0,0,\d_1)$ and $U(\mathfrak{g})(\chi_1,0,0,0)$.
Now we consider $U(\mathfrak{g})(\ep,0,0,\d)$ and $U(\mathfrak{g})(\ep,x,y,\d_2)$,
where $\d=0$ or $\d=\d_1$.
If $U(\mathfrak{g})(\ep,0,0,\d)$ and $U(\mathfrak{g})(\ep,x,y,\d_2)$
are isomorphic as generalized Hopf-Ore extensions of $U(\mathfrak{g})$,
then by \prref{1.11}, there exists an element $b\in U(\mathfrak{g})$ such that
$\D b= b\ot 1+1\ot b-x\ot y$. By \leref{2.2}, we have that $p>2$ and $kx=ky$,
which contradicts $kx\neq ky$. Hence $U(\mathfrak{g})(\ep,x,y,\d_2)$ is neither isomorphic to $U(\mathfrak{g})(\ep,0,0,0)$
nor isomorphic to $U(\mathfrak{g})(\ep,0,0,\d_1)$
as a generalized Hopf-Ore extension of $U(\mathfrak{g})$.
This completes the proof.
\end{proof}

For two generalized Hopf-Ore extensions in \prref{2.6}(b) or (d),
one can use \prref{1.11} to determine when they are isomorphic.

\subsection{The case of dim($\mathfrak{g}$)=2}\selabel{2.3}
It is well known that up to isomorphism, there are two $2$-dimensional
Lie algebras over $k$: the abelian Lie algebra $\mathfrak{g}_1=ka\op kb$
and the non-abelian Lie algebra $\mathfrak{g}_2=ka\op kb$ with $[a,b]=a$.
Obviously, $U(\mathfrak{g}_1)$ and $U(\mathfrak{g}_2)$ are not isomorphic as (Hopf) algebras.
Consequently, a generalized Hopf-Ore extension of $U(\mathfrak{g}_1)$ is not isomorphic to
a generalized Hopf-Ore extension of $U(\mathfrak{g}_2)$ by \prref{1.11}.

For any $\a, \b\in k$, one can define a character $\chi_{\a,\b}: U(\mathfrak{g}_1)\ra k$
by $\chi_{\a,\b}(a)=\a$ and $\chi_{\a,\b}(b)=\b$. Moreover, any character of $U(\mathfrak{g}_1)$
has the form $\chi_{\a,\b}$ for some $\a, \b\in k$. The algebra
automorphism $\tau_{\a,\b}$ of $U(\mathfrak{g}_1)$ induced by $\chi_{\a,\b}$
is given by $\tau_{\a,\b}(a)=a+\a$ and $\tau_{\a,\b}(b)=b+\b$  (see \prref{2.3}).
Let $\chi_1=\chi_{0,1}$ and $\tau_1=\tau_{0,1}$.
Note that $\chi_{0,0}=\ep$ and $\tau_{0,0}={\rm id}$.

Similarly, for any $\a\in k$, one can define a character $\chi_{\a}: U(\mathfrak{g}_2)\ra k$
by $\chi_{\a}(a)=0$ and $\chi_{\a}(b)=\a$. Moreover, any character of $U(\mathfrak{g}_2)$
has the form $\chi_{\a}$ for some $\a\in k$. The algebra
automorphism $\tau_{\a}$ of $U(\mathfrak{g}_2)$ induced by $\chi_{\a}$
is given by $\tau_{\a}(a)=a$ and $\tau_{\a}(b)=b+\a$.
Note that $\chi_0=\ep$ and $\tau_0={\rm id}$.

Let $GL(\mathfrak{g}_i)$ be the group of all linear automorphisms of $\mathfrak{g}_i$,
$i=1,2$. Under the basis $\{a, b\}$ of $\mathfrak{g}_i$, $GL(\mathfrak{g}_i)$ is isomorphic to
$GL_2(k)$, the group of all invertible $2\times 2$-matrices over $k$. For any
$A=\left(\begin{array}{cc}
             \a_{11} & \a_{12} \\
             \a_{21} & \a_{22} \\
           \end{array}
         \right)\in GL_2(k)$, the corresponding linear automorphism $\phi_A\in GL(\mathfrak{g}_i)$
is determined by $\phi_A(a, b)=(a, b)A$.
Since $\mathfrak{g}_1$ is abelian, $\phi_A$ is a Lie algebra automorphism of $\mathfrak{g}_1$
for any $A\in GL_2(k)$. It is easy to check that $\phi_A$ is a Lie algebra automorphism of $\mathfrak{g}_2$
if and only if $\a_{21}=0$ and $\a_{22}=1$.
Let $G_2$ be the subgroup of $GL_2(k)$ consisting of all matrices of form
$\left(\begin{array}{cc}
\a&\b\\
0&1\\
\end{array}\right)$ with $\a\in k^{\times}$ and $\b\in k$.

If $\phi_A$ is a Lie algebra automorphism of $\mathfrak{g}_i$, then $\phi_A$
can be uniquely extended to a Hopf algebra automorphism of $U(\mathfrak{g}_i)$,
denoted by $\Phi_A$. Conversely, assume that char$k=0$ and $\Phi$ is
a Hopf algebra automorphism of $U(\mathfrak{g}_i)$. Then it follows from \leref{2.1}
that the restriction $\Phi|_{\mathfrak{g}_i}$ is a Lie algebra automorphism of
$\mathfrak{g}_i$. Thus, we have the following lemma.

\begin{lemma}\lelabel{2.7}
Under the hypotheses above, we have

{\rm(a)} The map $GL_2(k)\ra {\rm Aut_{Hopf}}(U(\mathfrak{g}_1))$, $A\mapsto\Phi_A$
is a group monomorphism.\\
{\rm(b)} The map $G_2\ra {\rm Aut_{Hopf}}(U(\mathfrak{g}_2))$, $A\mapsto\Phi_A$
is a group monomorphism.\\
{\rm(c)} If char$k=0$, then the two maps in (1) and (2) are group isomorphisms.
\end{lemma}

\begin{lemma}\lelabel{2.8}
In $U(\mathfrak{g}_2)$, we have

{\rm(a)} $[a^n, b]=na^n$ and $[b^n,a]=\sum_{i=0}^{n-1}(-1)^{n-i}\binom{n}{i}ab^i$, $\forall n\>1$;\\
{\rm(b)} if char$k=p>0$ then $[a^{p^n},b]=\left\{\begin{array}{ll}
a,&n=0\\
0,&n\>1\\
\end{array}\right.$ and $[b^{p^n}, a]=(-1)^{p^n}a$, $\forall n\>0$.
\end{lemma}

\begin{proof}
(a) follows by induction on $n$, and (b) follows from (a).
\end{proof}

Now we discuss the generalized Hopf-Ore extensions on $U(\mathfrak{g}_i)$
in three cases: char$k=0$, char$k>2$ and char$k=2$, respectively.

{\bf Case 1}: char$k=0$. In this case, we only consider the generalized Hopf-Ore extensions of $U(\mathfrak{g}_2)$.
The generalized Hopf-Ore extensions of $U(\mathfrak{g}_1)$ will be considered in \seref{2.4}.
Let $H$ be a generalized Hopf-Ore extension on $U(\mathfrak{g}_2)$.
Then by \prref{2.3}, we may assume that $H=U(\mathfrak{g}_2)(\chi_{\a}, x, y, \d)$,
where $\a\in k$, $x,y\in\mathfrak{g}_2$
and $\d$ is a $\tau_{\a}$-derivation of $U(\mathfrak{g}_2)$ such that \equref{e10} is satisfied for all $c\in\mathfrak{g}$.
If $x=y=0$ or $kx=ky\neq 0$, then by \coref{1.12}(d), $H\cong U(\mathfrak{g}_2)(\chi_{\a},0,0,\d)$
with $\d(\mathfrak{g}_2)\subseteq\mathfrak{g}_2$, a usual Hopf-Ore extension of $U(\mathfrak{g}_2)$.
In case that $x, y\in \mathfrak{g}$ are linearly independent, we have
$(x, y)=(a,b)A$ for some $A=\left(
               \begin{array}{cc}
                 \a_{11} & \a_{12} \\
                 \a_{21} & \a_{22} \\
               \end{array}
             \right)\in GL_2(k)$.
By \equref{e10}, we have
$$\begin{array}{rl}
  \D(\d(b))=&\d(b)\ot 1+1\ot \d(b)+[x,b]\ot y+x\otimes[y,b]-\chi_{\a}(b)x\ot y\\
 =&\d(b)\ot 1+1\ot \d(b)+(2-\a)\a_{11}\a_{12}a\otimes a-\a\a_{21}\a_{22}b\otimes b\\
 &+(1-\a)\a_{11}\a_{22}a\otimes b+(1-\a)\a_{12}\a_{21}b\otimes a.
\end{array}$$
By \leref{2.2}(a), $(1-\a)\a_{11}\a_{22}=(1-\a)\a_{12}\a_{21}$, and so $\a=1$.
Let $\lambda={\rm det}(A)^{-1}\in k^{\times}$ and
$b'=\frac{1}{2}\lambda\a_{11}\a_{12}a^2+\lambda\a_{12}\a_{21}ab+\frac{1}{2}\lambda\a_{21}\a_{22}b^2\in U(\mathfrak{g}_2)$.
Then a straightforward computation shows that
$\D(b')=b'\otimes 1+1\otimes b'+\lambda x\otimes y-a\otimes b$.
Let $\d'': U(\mathfrak{g}_2)\ra U(\mathfrak{g}_2)$ be defined by
$\d''(u)=\tau_1(u)b'-b'u$ for all $u\in U(\mathfrak{g}_2)$, and let $\d'=\lambda\d+\d''$.
Then it follows from \prref{1.11} that $U(\mathfrak{g}_2)(\chi_1, x, y, \d)$
is isomorphic, as a generalized Hopf-Ore extension of $U(\mathfrak{g}_2)$,
to $U(\mathfrak{g}_2)(\chi_1, a, b, \d')$.
Again by \equref{e10}, we have
$$\begin{array}{c}
 \D(\d'(a))=\d'(a)\ot 1+1\ot \d'(a)-a\otimes a,\
 \D(\d'(b))=\d'(b)\ot 1+1\ot \d'(b).
\end{array}$$
Hence $\d'(b)\in \mathfrak{g}_2$. By \leref{2.2}(a), $\d'(a)+\frac{1}{2}a^2\in \mathfrak{g}_2$.

\begin{proposition}\prlabel{2.10}
Assume that char$(k)=0$. Then up to isomorphism, there are two classes of generalized Hopf-Ore extensions of
$U(\mathfrak{g}_2)$ as follows:

{\rm(a)} $U(\mathfrak{g}_2)(\chi_{\a},0,0,\d_1)$, where $\a\in k$ and
  $\d_1$ is a $\tau_{\a}$-derivation of $U(\mathfrak{g}_2)$
  with $\d_1(\mathfrak{g}_2)\subseteq\mathfrak{g}_2$;\\
{\rm(b)} $U(\mathfrak{g}_2)(\chi_1,a,b,\d_2)$, where $\d_2$ is a $\tau_1$-derivation of $U(\mathfrak{g}_2)$
  with $\d_2(a)+\frac{1}{2}a^2\in\mathfrak{g}_2$ and $\d_2(b)\in\mathfrak{g}_2$.

Moreover, $U(\mathfrak{g}_2)(\chi_{\a},0,0,\d_1)$ and $U(\mathfrak{g}_2)(\chi_1,a,b,\d_2)$
are not isomorphic generalized Hopf-Ore extensions of $U(\mathfrak{g}_2)$.
Furthermore, $U(\mathfrak{g}_2)(\chi_{\a},0,0,\d_1)$
is isomorphic, as a Hopf algebra, to the enveloping algebras of some 3-dimensional Lie algebra.
\end{proposition}
\begin{proof}
The first claim follows from the discussion above, the second one
 follows from \prref{1.11}, and the last one is obvious.
\end{proof}

Using \prref{1.11}, it is easy to determine when two generalized Hopf-Ore extensions
in the same class of \prref{2.10} are isomorphic.

{\bf Case 2}: char$k=p>2$. Firstly, let $H$ be a generalized Hopf-Ore extension of $U(\mathfrak{g}_1)$.
By \prref{2.3}, we may assume that $H=U(\mathfrak{g}_1)(\chi_{\a,\b},x,y,\d)$,
where $\a, \b\in k$, $x, y\in\hat{\mathfrak{g}}_1$
and $\d$ is a $\tau_{\a,\b}$-derivation of $U(\mathfrak{g}_1)$ such that \equref{e10} is satisfied for all $c\in\mathfrak{g}$.

If $x=y=0$ or $kx=ky\neq 0$,
then by \coref{1.12}(d), $H$ is isomorphic to $U(\mathfrak{g}_1)(\chi_{\a,\b},0,0,\d)$ with
$\d(\mathfrak{g}_1)\subseteq\hat{\mathfrak{g}}_1$. If $(\a,\b)\neq(0,0)$, then one can choose a matrix
$A\in GL_2(k)$ such that $(\a,\b)=(0,1)A$. In this case, $\chi_{\a,\b}\Phi_{A^{-1}}=\chi_1$.
By \coref{1.12}(e) and \leref{2.7}(a),
$U(\mathfrak{g}_1)(\chi_{\a,\b},0,0,\d)\cong U(\mathfrak{g}_1)(\chi_1,0,0,\Phi_A\d\Phi_{A^{-1}})$.
Hence $H$ is isomorphic to $U(\mathfrak{g}_1)(\ep,0,0,\d)$
or $U(\mathfrak{g}_1)(\chi_1,0,0, \d)$ as a generalized Hopf-Ore extension of $U(\mathfrak{g}_1)$,
where $\d$ is a derivation or a $\tau_1$-derivation of $U(\mathfrak{g}_1)$
with $\d(\mathfrak{g}_1)\subseteq\hat{\mathfrak{g}}_1$. If
$x, y\in\hat{\mathfrak g}_1$ are linearly independent, then by \equref{e10} we have
\begin{eqnarray*}
&&\d(a)=\d(a)\ot 1+1\ot\d(a)-\a x\ot y,\\
&&\d(b)=\d(b)\ot 1+1\ot\d(b)-\b x\ot y.\\
\end{eqnarray*}
Thus, by \leref{2.2}(a), $\a=\b=0$,
and so $H=U(\mathfrak{g}_1)(\ep,x,y,\d)$ with $\d(\mathfrak{g}_1)\subseteq\hat{\mathfrak{g}}_1$.
Summarizing the discussion above, we have the following proposition.

\begin{proposition}\prlabel{2.12}
Assume that char$(k)=p>2$. Then up to isomorphism, there are three classes of generalized Hopf-Ore extensions of
$U(\mathfrak{g}_1)$ as follows:

{\rm(a)} $U(\mathfrak{g}_1)(\ep,0,0,\d_1)$, where $\d_1$ is a derivation of $U(\mathfrak{g}_1)$
  with $\d_1(\mathfrak{g}_1)\subseteq\hat{\mathfrak g}_1$;\\
{\rm(b)} $U(\mathfrak{g}_1)(\chi_1,0,0,\d_2)$, where $\d_2$ is a $\tau_1$-derivation of $U(\mathfrak{g}_1)$
  with $\d_2(\mathfrak{g}_1)\subseteq\hat{\mathfrak g}_1$;\\
{\rm(c)} $U(\mathfrak{g}_1)(\ep,x,y,\d_3)$, where $0\neq x, y\in\hat{\mathfrak g}_1$ with $kx\neq ky$,
  and $\d_3$ is a derivation of $U(\mathfrak{g}_1)$
  with $\d_3({\mathfrak g}_1)\subseteq\hat{\mathfrak g}_1$.

Moreover, as generalized Hopf-Ore extensions,
$U(\mathfrak{g}_1)(\ep,0,0,\d_1)$, $U(\mathfrak{g}_1)(\chi_1,0,0,\d_2)$
and $U(\mathfrak{g}_1)(\ep,x,y,\d_3)$ are pairwise non-isomorphic.
\end{proposition}
\begin{proof}
The first claim follows from the discussion above, and the second one follows from \prref{1.11}.
\end{proof}

Next, let $H$ be a generalized Hopf-Ore extension of $U(\mathfrak{g}_2)$.
Then similarly, we may assume that $H=U(\mathfrak{g}_2)(\chi_{\a}, x, y, \d)$,
where $\a\in k$, $x,y\in\hat{\mathfrak g}_2$
and $\d$ is a $\tau_{\a}$-derivation of $U(\mathfrak{g}_2)$
such that \equref{e10} is satisfied for all $c\in\mathfrak{g}$.
If $x=y=0$ or $kx=ky\neq 0$, then $H\cong U(\mathfrak{g}_2)(\chi_{\a},0,0,\d)$
with $\d({\mathfrak g}_2)\subseteq\hat{\mathfrak g}_2$, a usual Hopf-Ore extension of $U(\mathfrak{g}_2)$.
Now assume that $x, y\in\hat{\mathfrak g}_2$ are linearly independent. Then
$x=\sum_{i\>0}(\a_ia^{p^i}+\a'_ib^{p^i})$ and
$y=\sum_{i\>0}(\b_ia^{p^i}+\b'_ib^{p^i})$ for some almost all zero
elements $\a_i,\a'_i,\b_i,\b'_i\in k$. Let $\a'=-\sum_{i\>0}\a_i'$ and $\b'=-\sum_{i\>0}\b_i'$.
Then by \equref{e10} and \leref{2.8}, we have
\begin{eqnarray*}
  &&\D(\d(a))=\d(a)\ot 1+1\ot \d(a)+\a' a\ot y+\b'x\ot a, \\
 &&\D(\d(b))=\d(b)\ot 1+1\ot \d(b)+\a_0a\ot y+\b_0x\ot a-\a x\ot y.
\end{eqnarray*}
By \leref{2.2}(a), we have the following equations (*):
$$\begin{array}{ll}
\a'\b_i=\b'\a_i,\ \forall i\>1;& \a'\b'_i=\b'\a'_i,\ \forall i\>0;\\
(1-\a)\a_0\b_i=(1-\a)\b_0\a_i,\ \forall i\>1;& \a\a_i\b_j=\a\a_j\b_i,\ \forall j>i\>1;\\
(1-\a)\a_0\b'_i=(1-\a)\b_0\a'_i,\ \forall i\>0;& \a\a_i\b'_j=\a\a'_j\b_i,\ \forall i\>1, j\>0;\\
\a\a'_i\b'_j=\a\a'_j\b'_i,\ \forall j>i\>0.&\\
\end{array}$$
Suppose $\a\neq 1$. If $\a_0\neq 0$, then $\b_0=\g\a_0$ for some $\g\in k$.
Thus, from the equations above one gets that $\b_i=\g\a_i$, $\forall i\>1$ and $\b'_i=\g\a'_i$, $\forall i\>0$.
Hence $y=\g x$, a contradiction. This shows that $\a_0=0$. Then we have
$\b_0\a_i=\b_0\a'_i=0$, $\forall i\>0$, and so $\b_0=0$ since $x\neq 0$.
Furthermore, suppose $\a\neq 0$. Then from the equations above, one can see that
$x$ and $y$ are linearly dependent over $k$, a contradiction. Thus, we have proven that either $\a=0$ or $\a=1$,
and that $\a_0=\b_0=0$ when $\a=0$.

In case $\a=0$, $\a_0=\b_0=0$, and the equations (*) become
$$\begin{array}{ll}
\a'\b_i=\b'\a_i,\ \forall i\>1;& \a'\b'_i=\b'\a'_i,\ \forall i\>0.\\
\end{array}$$
Then a similar argument as above shows that $\a'=\b'=0$.
Hence $\d(a), \d(b)\in\hat{\mathfrak g}_2$.

In case $\a=1$, the equations (*) become
$$\begin{array}{lll}
\a'\b_i=\b'\a_i,\ \forall i\>1;& \a'\b'_i=\b'\a'_i,\ \forall i\>0;&\\
\a_i\b_j=\a_j\b_i,\ \forall j>i\>1;& \a_i\b'_j=\a'_j\b_i,\ \forall i\>1, j\>0;&
\a'_i\b'_j=\a'_j\b'_i,\ \forall j>i\>0.\\
\end{array}$$
The above equations are equivalent to that $x-\a_0a$ and $y-\b_0a$ are linearly dependent.
Moreover, either $x-\a_0a\neq 0$ or $y-\b_0a\neq 0$ since $x$ and $y$ are linearly independent.

If $x-\a_0a=0$, then $y-\b_0a\neq 0$, $\a'=0$, and $\a_0\neq 0$ by $x\neq 0$.
In this case, we may assume $x=a$ by \coref{1.12}(a), i.e., $\a_0=1$. Then by \leref{2.2}(a),
$\d(a)-\frac{1}{2}\b'a^2\in\hat{\mathfrak g}_2$ and $\d(b)-\frac{1}{2}\b_0a^2\in\hat{\mathfrak g}_2$.
Similarly, if $y-\b_0a=0$, then $x-\a_0a\neq 0$, $\b'=0$, $y=a$,
$\d(a)-\frac{1}{2}\a'a^2\in\hat{\mathfrak g}_2$ and $\d(b)-\frac{1}{2}\a_0a^2\in\hat{\mathfrak g}_2$.
Now suppose that $x-\a_0a\neq 0$ and $y-\b_0a\neq 0$. Then $y-\b_0a=\l(x-\a_0a)$ for some
$\l\in k^{\times}$. By \coref{1.12}(a), we may assume $\l=1$.
Then $y-\b_0a=x-\a_0a$ and $\b'=\a'$. Since $x$ and $y$ are linearly independent, $\b_0\neq\a_0$.
Hence $y=(\b_0-\a_0)a+x$, and so $(\b_0-\a_0)^{-1}y=a+(\b_0-\a_0)^{-1}x$.
Again by \coref{1.12}(a), we may assume $\b_0-\a_0=1$.
Then $y=a+x$, and so
\begin{eqnarray*}
  \D(\d(a))&=&\d(a)\ot 1+1\ot \d(a)+\a' a\ot a+\a'(a\ot x+x\ot a),\\
 \D(\d(b))&=&\d(b)\ot 1+1\ot \d(b)+\a_0a\ot a+\a_0(a\ot x+x\ot a)-x\ot x.\\
\end{eqnarray*}
Thus, by \leref{2.2}(a), we have $\d(a)-\frac{1}{2}\a'a^2-\a'ax\in\hat{\mathfrak g}_2$
and $\d(b)-\frac{1}{2}\a_0a^2-\a_0ax+\frac{1}{2}x^2\in\hat{\mathfrak g}_2$.

Summarizing the discussion above, we have the following proposition.

\begin{proposition}\prlabel{2.13}
Assume that char$(k)=p>2$. Then each generalized Hopf-Ore extensions of
$U(\mathfrak{g}_2)$ is isomorphic to one of the followings:

{\rm(a)} $U(\mathfrak{g}_2)(\chi_{\a},0,0,\d_1)$, where $\d_1$ is a $\tau_{\a}$-derivation of $U(\mathfrak{g}_2)$
  with $\d_1(\mathfrak{g}_2)\subseteq\hat{\mathfrak g}_2$;\\
{\rm(b)} $U(\mathfrak{g}_2)(\ep,x,y,\d_2)$, where $x=\sum_{i\>1}\a_ia^{p^i}+\sum_{i\>0}\a'_ib^{p^i}\neq 0$ and
$y=\sum_{i\>1}\b_ia^{p^i}+\sum_{i\>0}\b'_ib^{p^i}\neq 0$ for some almost all zero
elements $\a_i,\a'_i,\b_i,\b'_i\in k$ with $kx\neq ky$ and $\sum_{i\>0}\a_i'=\sum_{i\>0}\b_i'=0$,
$\d_2$ is a derivation of $U(\mathfrak{g}_2)$
  with $\d_2(\mathfrak{g}_2)\subseteq\hat{\mathfrak g}_2$;\\
{\rm(c)} $U(\mathfrak{g}_2)(\chi_1,x,y,\d_3)$, where $0\neq x, y\in\hat{\mathfrak g}_2$,
and $\d_3$ is a $\tau_1$-derivation of $U(\mathfrak{g}_2)$ such that one of the followings is satisfied:
\begin{enumerate}
  \item[(1)] $x=a$, $y=\sum_{i\>0}(\b_ia^{p^i}+\b'_ib^{p^i})\neq\b_0a$ for some almost all zero
elements $\b_i,\b'_i\in k$,
$\d_3(a)-\frac{1}{2}\b'a^2\in\hat{\mathfrak g}_2$, $\d_3(b)-\frac{1}{2}\b_0a^2\in\hat{\mathfrak g}_2$,
where $\b'=-\sum_{i\>0}\b'_i$;
  \item[(2)] $y=a$, $x=\sum_{i\>0}(\a_ia^{p^i}+\a'_ib^{p^i})\neq\a_0a$
for some almost all zero elements $\a_i,\a'_i\in k$,
$\d_3(a)-\frac{1}{2}\a'a^2\in\hat{\mathfrak g}_2$, $\d_3(b)-\frac{1}{2}\a_0a^2\in\hat{\mathfrak g}_2$,
where $\a'=-\sum_{i\>0}\a_i'$;
  \item[(3)] $x=\sum_{i\>0}(\a_ia^{p^i}+\a'_ib^{p^i})\neq \a_0 a$ for some almost all zero elements
   $\a_i,\a'_i\in k$, $y=a+x$, $\d_3(a)-\frac{1}{2}\a'a^2-\a'ax\in\hat{\mathfrak g}_2$
and $\d_3(b)-\frac{1}{2}\a_0a^2-\a_0ax+\frac{1}{2}x^2\in\hat{\mathfrak g}_2$,
where $\a'=-\sum_{i\>0}\a'_i$.
         \end{enumerate}
Moreover, as generalized Hopf-Ore extensions of $U(\mathfrak{g}_2)$,
$U(\mathfrak{g}_2)(\chi_{\a},0,0,\d_1)$, $U(\mathfrak{g}_2)(\ep,x,y,\d_2)$
and $U(\mathfrak{g}_2)(\chi_1,x,y,\d_3)$ are pairwise non-isomorphic.
\end{proposition}
\begin{proof}
The first statement follows from the discussion above, and the second one follows from \prref{1.11}.
\end{proof}

{\bf Case 3}: char$k=2$. Firstly, let $H$ be a generalized Hopf-Ore extension on $U(\mathfrak{g}_1)$.
By \prref{2.3}, we may assume that $H=U(\mathfrak{g}_1)(\chi_{\a,\b},x,y,\d)$,
where $\a, \b\in k$, $x, y\in\hat{\mathfrak{g}}_1$
and $\d$ is a $\tau_{\a,\b}$-derivation of $U(\mathfrak{g}_1)$ such that \equref{e10} is satisfied for all $c\in\mathfrak{g}_1$.
If $x=y=0$, then an argument similar to Case 2 shows that $H$ is isomorphic to $U(\mathfrak{g}_1)(\ep, 0, 0, \d)$
or $U(\mathfrak{g}_1)(\chi_1, 0, 0, \d)$, where $\d$ is a derivation or $\tau_1$-derivation of $U(\mathfrak{g}_1)$
with $\d(\mathfrak{g}_1)\subseteq\hat{\mathfrak g}_1$.
If $x\neq 0$ and $y\neq 0$, then by \equref{e10}, we have
\begin{eqnarray*}
  &&\D(\d(a))=\d(a)\ot 1+1\ot \d(a)-\a x\ot y, \\
 &&\D(\d(b))=\d(b)\ot 1+1\ot \d(b)-\b x\ot y.
\end{eqnarray*}
By \leref{2.2}(b), $\a=\b=0$. Hence $\d(a), \d(b)\in\hat{\mathfrak g}_1$.
Thus, we have the following proposition.

\begin{proposition}\prlabel{2.14}
Assume that char$(k)=2$. Then each generalized Hopf-Ore extensions on
$U(\mathfrak{g}_1)$ is isomorphic to one of the followings:

{\rm(a)} $U(\mathfrak{g}_1)(\ep,0,0,\d_1)$, where $\d_1$ is a derivation of $U(\mathfrak{g}_1)$
  with $\d_1(\mathfrak{g}_1)\subseteq\hat{\mathfrak g}_1$;\\
{\rm(b)} $U(\mathfrak{g}_1)(\chi_1,0,0,\d_2)$, where $\d_2$ is a $\tau_1$-derivation of $U(\mathfrak{g}_1)$
  with $\d_2(\mathfrak{g}_1)\subseteq\hat{\mathfrak g}_1$;\\
{\rm(c)} $U(\mathfrak{g}_1)(\ep,x,y,\d_3)$, where $0\neq x, y\in\hat{\mathfrak g}_1$,
  and $\d_3$ is a derivation of $U(\mathfrak{g}_1)$
  with $\d_3({\mathfrak g}_1)\subseteq\hat{\mathfrak g}_1$.

Moreover, $U(\mathfrak{g}_1)(\ep,0,0,\d_1)$, $U(\mathfrak{g}_1)(\chi_1,0,0,\d_2)$
and $U(\mathfrak{g}_1)(\ep,x,y,\d_3)$ are pairwise non-isomorphic
generalized Hopf-Ore extensions on $U(\mathfrak{g}_1)$.
\end{proposition}

\begin{proof}
The first statement follows from the discussion above, and the second one follows from \prref{1.11}.
\end{proof}

Next, let $H$ be a generalized Hopf-Ore extension on $U(\mathfrak{g}_2)$.
Then similarly, we may assume that $H=U(\mathfrak{g}_2)(\chi_{\a}, x, y, \d)$,
where $\a\in k$, $x,y\in\hat{\mathfrak g}_2$
and $\d$ is a $\tau_{\a}$-derivation of $U(\mathfrak{g}_2)$ such that \equref{e10} is satisfied for all $c\in\mathfrak{g}$.
If $x=y=0$, then $H\cong U(\mathfrak{g}_2)(\chi_{\a},0,0,\d)$
with $\d({\mathfrak g}_2)\subseteq\hat{\mathfrak g}_2$, a usual Hopf-Ore extension on $U(\mathfrak{g}_2)$.
Now assume $x\neq0$ and $y\neq 0$. Then
$x=\sum_{i\>0}(\a_ia^{2^i}+\a'_ib^{2^i})$ and
$y=\sum_{i\>0}(\b_ia^{2^i}+\b'_ib^{2^i})$ for some almost all zero
elements $\a_i,\a'_i,\b_i,\b'_i\in k$. Let $\a'=\sum_{i\>0}\a_i'$ and $\b'=\sum_{i\>0}\b_i'$.
Then by \equref{e10} and \leref{2.8}, we have
\begin{eqnarray}
  &&\D(\d(a))=\d(a)\ot 1+1\ot \d(a)+\a' a\ot y+\b'x\ot a, \eqlabel{daeq}\\
 &&\D(\d(b))=\d(b)\ot 1+1\ot \d(b)+\a_0a\ot y+\b_0x\ot a-\a x\ot y. \eqlabel{dbeq}
\end{eqnarray}
By \leref{2.2}(b), we have the following equations (**):
$$\begin{array}{ll}
\a'\b_i=\b'\a_i,\ \forall i\>0;& \a'\b'_i=\b'\a'_i,\ \forall i\>0;\\
(1-\a)\a_0\b_i=(1-\a)\b_0\a_i,\ \forall i\>1;& \a\a_i\b_j=\a\a_j\b_i,\ \forall j>i\>1;\\
(1-\a)\a_0\b'_i=(1-\a)\a'_i\b_0,\ \forall i\>0;& \a\a_i\b'_j=\a\a'_j\b_i,\ \forall i\>1, j\>0;\\
\a\a'_i\b'_i=\a\a_i\b_i=0,\ \forall i\>0;& \a\a'_i\b'_j=\a\a'_j\b'_i,\ \forall j>i\>0.\\
\end{array}$$

Suppose $\a\neq 0$ and $\a\neq 1$. If $\a_0\neq 0$, then $\b_0=0$ by $\a\a_0\b_0=0$.
Then from $(1-\a)\a_0\b_i=(1-\a)\b_0\a_i$, $\forall i\>1$ and
$(1-\a)\a_0\b'_i=(1-\a)\a'_i\b_0$, $\forall i\>0$, one gets $\b_i=0$, $\forall i\>1$
and $\b'_i=0$, $\forall i\>0$.
Hence $y=0$, a contradiction. Thus, $\a_0=0$. Similarly, $\b_0=0$.
Since $x\neq0$, there exists an integer $i_0\>1$ or $i_0\>0$ such that $\a_{i_0}\neq 0$
or $\a'_{i_0}\neq 0$. If $\a_{i_0}\neq 0$ for some $i_0\>1$,
then $\b_{i_0}=0$ by $\a\a_{i_0}\b_{i_0}=0$.
Then from $\a\a_i\b_j=\a\a_j\b_i$, $\forall j>i\>1$ and
$\a\a_i\b'_j=\a\a'_j\b_i$, $\forall i\>1, j\>0$,
one gets $\b_i=0$, $\forall i_0\neq i\>1$
and $\b'_i=0$, $\forall i\>0$. Hence $y=0$, a contradiction.
Similarly, if $\a'_{i_0}\neq 0$ for some $i_0\>0$, then $y=0$, a contradiction.
Thus, we have proven that either $\a=0$ or $\a=1$.

 In case $\a=0$, the equations (**) become
$$\begin{array}{ll}
\a'\b_i=\b'\a_i,\ \forall i\>1;& \a'\b'_i=\b'\a'_i,\ \forall i\>0,\\
\a_0\b_i=\b_0\a_i,\ \forall i\>1;&\a_0\b'_i=\a'_i\b_0,\ \forall i\>0.\\
\end{array}$$
If $kx=ky$, then $y=\g x$ for some $\g\in k^{\times}$.
In this case, we may assume that $y=x$ by \coref{1.12}(a).
Moreover, $\d(a)-\a' ax\in\hat{\mathfrak g}_2$ and
$\d(b)-\a_0ax\in\hat{\mathfrak g}_2$.
Now assume that $kx\neq ky$.
If $\a_0\neq 0$, then $\b_0=\g\a_0$ for some $\g\in k$.
Then from $\a_0\b_i=\b_0\a_i$, $\forall i\>1$ and $\a_0\b'_i=\a'_i\b_0$, $\forall i\>0$,
one gets $\b_i=\g\a_i$, $\forall i\>1$ and $\b'_i=\g\a'_i$, $\forall i\>0$.
This implies $y=\g x$. By $x, y\neq 0$, $\g\neq 0$ and $kx=ky$, a contradiction.
Hence $\a_0=0$. Similarly, we have $\b_0=0$.
Then from $\a'\b_i=\b'\a_i$, $\forall i\>1$ and $\a'\b'_i=\b'\a'_i$, $\forall i\>0$,
a similar argument as above shows that $\a'=\b'=0$.
In this case, $\d(a)\in\hat{\mathfrak g}_2$ and
$\d(b)\in\hat{\mathfrak g}_2$.

In case $\a=1$, the equations (**) become
$$\begin{array}{ll}
\a'\b_i=\b'\a_i,\ \forall i\>0;& \a'\b'_i=\b'\a'_i,\ \forall i\>0;\\
\a_i\b_j=\a_j\b_i,\ \forall j>i\>1;& \a_i\b'_j=\a'_j\b_i,\ \forall i\>1, j\>0;\\
\a'_i\b'_i=\a_i\b_i=0,\ \forall i\>0;& \a'_i\b'_j=\a'_j\b'_i,\ \forall j>i\>0.\\
\end{array}$$
These equations implies that
$x-\a_0a$ and $y-\b_0a$ are linearly dependent.
If $x-\a_0a=y-\b_0a=0$, then $x=\a_0a$ and $y=\b_0a$.
Since $\a_0\b_0=0$, $\a_0=0$ or $\b_0=0$, and hence $x=0$ or $y=0$, a contradiction.
It follows that either $x-\a_0a\neq 0$ or $y-\b_0a\neq0$.

If $x-\a_0a\neq 0$, then $y-\b_0a=\g(x-\a_0a)$ for some $\g\in k$,
and there is an integer $i_0\>1$ (or $i_0\>0$) such that $\a_{i_0}\neq0$ (or $\a'_{i_0}\neq0$).
Hence $\b_{i_0}=0$ (or $\b'_{i_0}=0$) by $\a_{i_0}\b_{i_0}=0$ (or $\a'_{i_0}\b'_{i_0}=0$).
However, $\b_{i_0}=\g\a_{i_0}$ (or $\b'_{i_0}=\g\a'_{i_0}$), hence $\g=0$. Thus, $y=\b_0a$ and $\b_0\neq 0$,
which implies $\a_0=0$ by $\a_0\b_0=0$. In this case, we may assume $y=a$ by \coref{1.12}(a),
i.e., $\b_0=1$. Hence $\b'=0$, and $\a'=0$ by $\a'\b_0=\b'\a_0$.
Then by \leref{2.2}(b), $\d(a), \d(b)\in\hat{\mathfrak g}_2$.

Similarly, if $y-\b_0a\neq 0$, then $x-\a_0a=0$. We may assume $x=a$. Moreover, we have
$\a'=\b_0=\b'=0$ and $\d(a), \d(b)\in\hat{\mathfrak g}_2$.
In this case, $H\cong U(\mathfrak{g}_2)(\chi_1, a, y, \d)$.
We claim that $U(\mathfrak{g}_2)(\chi_1, a, y, \d)\cong U(\mathfrak{g}_2)(\chi_1, y, a, \d)$
as generalized Hopf-Ore extensions on $U(\mathfrak{g}_2)$.
In fact, let $b'=ay$. Then
$\D(b')=\D(a)\D(y)=(a\ot 1+1\ot a)(y\ot 1+1\ot y)=ay\ot 1+1\ot ay+a\ot y+y\ot a
=b'\ot 1+1\ot b'+a\ot y-y\ot a$.
Define a map $\d'': U(\mathfrak{g}_2)\ra U(\mathfrak{g}_2)$ by
$\d''(u)=\tau_1(u)b'-b'u$ for any $u\in U(\mathfrak{g}_2)$.
Then $\d''$ is a $\tau_1$-derivation of $U(\mathfrak{g}_2)$.
By \leref{2.8}(b), a straightforward computation shows that
$\d''(a)=\d''(b)=0$. Hence $\d''=0$, and so $\d+\d''=\d$.
Thus, it follows from \prref{1.11} that $U(\mathfrak{g}_2)(\chi_1, a, y, \d)$
and $U(\mathfrak{g}_2)(\chi_1, y, a, \d)$ are isomorphic
generalized Hopf-Ore extensions on $U(\mathfrak{g}_2)$.

Summarizing the discussion above, we have the following proposition.

\begin{proposition}\prlabel{2.15}
Assume char$(k)=2$. Then each generalized Hopf-Ore extension on $U(\mathfrak{g}_2)$
is isomorphic to one in the followings:

{\rm (a)} $U(\mathfrak{g}_2)(\chi_{\a},0,0,\d_1)$, where $\a\in k$ and $\d_1$ is a $\tau_{\a}$-derivation
 of $U(\mathfrak{g}_2)$ with $\d_1(\mathfrak{g}_2)\subseteq\hat{\mathfrak{g}}_2$.\\
{\rm (b)} $U(\mathfrak{g}_2)(\ep,x,x,\d_2)$, where $x=\sum_{i\>0}(\a_ia^{2^i}+\a'_ib^{2^i})\neq 0$
for some almost all zero elements $\a_i, \a'_i\in k$, $\a'=\sum_{i\>0}\a'_i$,
and $\d_2$ is a derivation of $U(\mathfrak{g}_2)$ with
$\d_2(a)-\a'ax\in\hat{\mathfrak{g}}_2$ and $\d_2(b)-\a_0ax\in\hat{\mathfrak{g}}_2$.\\
{\rm (c)} $U(\mathfrak{g}_2)(\ep,x,y,\d_3)$, where $x=\sum_{i\>1}\a_ia^{2^i}+\sum_{i\>0}\a'_ib^{2^i}\neq 0$,
$y=\sum_{i\>1}\b_ia^{2^i}+\sum_{i\>0}\b'_ib^{2^i}\neq 0$ for some almost all zero elements
$\a_i,\a'_i,\b_i,\b'_i\in k$ with $kx\neq ky$ and $\sum_{i\>0}\a_i'=\sum_{i\>0}\b_i'=0$, and
$\d_3$ is a derivation of $U(\mathfrak{g}_2)$ with $\d_3(\mathfrak{g}_2)\subseteq\hat{\mathfrak g}_2$.\\
{\rm (d)} $U(\mathfrak{g}_2)(\chi_1,x,a,\d_4)$, where $x=\sum_{i\>1}\a_ia^{2^i}+\sum_{i\>0}\a'_ib^{2^i}\neq 0$
for some almost all zero elements
$\a_i,\a'_i\in k$ with $\sum_{i\>0}\a_i'=0$, and
$\d_4$ is a $\tau_1$-derivation of $U(\mathfrak{g}_2)$ with $\d_4(\mathfrak{g}_2)\subseteq\hat{\mathfrak g}_2$.

Moreover, $U(\mathfrak{g}_2)(\chi_{\a},0,0,\d_1)$, $U(\mathfrak{g}_2)(\ep,x,x,\d_2)$,
$U(\mathfrak{g}_2)(\ep,x,y,\d_3)$ and $U(\mathfrak{g}_2)(\chi_1,x,a,\d_4)$
are pairwise non-isomorphic generalized Hopf-Ore extensions on $U(\mathfrak{g}_2)$.
\end{proposition}

\begin{proof}
The first statement follows from the discussion above, and the second one follows from \prref{1.11}.
\end{proof}

\subsection{The case of dim($\mathfrak{g}$)=$n$ with $n\>$2 and char($k$)=0}\selabel{2.4}
Throughout this subsection, assume that char$(k)=0$ and $\mathfrak{g}$ is an $n$-dimensional
abelian Lie algebra. Let $\{a_1, a_2, \cdots, a_n\}$ be a fixed basis of $\mathfrak g$ over $k$.

For any $\a=(\a_1, \a_2, \cdots, \a_n)\in k^n$, one can define a character $\chi_{\a}: U(\mathfrak{g})\ra k$
by $\chi_{\a}(a_i)=\a_i$, $1\<i\<n$. Moreover, any character of $U(\mathfrak{g})$
is equal to some  $\chi_{\a}$, $\a\in k^n$. The algebra
automorphism $\tau_{\a}$ of $U(\mathfrak{g})$ induced by $\chi_{\a}$
is given by $\tau_{\a}(a_i)=a_i+\a_i$, $1\<i\<n$ (see \prref{2.3}).
Note that $\chi_{0}=\ep$ and $\tau_{0}={\rm id}$.
Let ${\bf 1}=(1, 0, \cdots, 0)\in k^n$.

Let $M_n(k)$ be the algebra of all $n\times n$-matrices over $k$,
and $GL_n(k)$ the group of all invertible matrices in $M_n(k)$. For any $A=(\a_{ij})\in M_n(k)$,
one can define a linear endomorphism $\phi_A\in{\rm End}_k(\mathfrak g)$ by
$\phi_A(a_i)=\sum_{j=1}^n\a_{ji}a_j$, $1\<i\<n$. Moreover, the map
$\phi: M_n(k)\ra{\rm End}_k(\mathfrak g)$, $A\mapsto\phi_A$, is an algebra isomorphism.
Let $GL(\mathfrak{g})$ be the group of all linear automorphisms of $\mathfrak{g}$.
Then $A\in GL_n(k)$ if and only if $\phi_A\in GL(\mathfrak{g})$.
Since $\mathfrak{g}$ is abelian, $\phi_A$ is a Lie algebra endomorphism (or automorphism) of $\mathfrak{g}$
for any $A\in M_n(k)$ (or $GL_n(k)$). Therefore $\phi_A$
can be uniquely extended to a Hopf algebra endomorphism (or automorphism) $\Phi_A$ of $U(\mathfrak{g})$.
Conversely, since char$k=0$, any Hopf algebra endomorphism (or automorphism)
of $U(\mathfrak{g})$ is equal to some $\Phi_A$, $A\in M_n(k)$ (or $A\in GL_n(k)$).
Thus, we have the following lemma.

\begin{lemma}\lelabel{2.16}
Under the hypotheses above, we have

{\rm (a)} The map $\phi: M_n(k)\ra {\rm End}_k(\mathfrak{g})$, $A\mapsto\phi_A$
is an algebra isomorphism.\\
{\rm (b)} The map $\Phi: GL_n(k)\ra {\rm Aut_{Hopf}}(U(\mathfrak{g}))$, $A\mapsto\Phi_A$
is a group isomorphism.
\end{lemma}

Let $\a\in k^n$. Since $U(\mathfrak g)$ is generated as an algebra by $\mathfrak g$,
a $\tau_{\a}$-derivation of $U(\mathfrak g)$ is determined by its value on $\mathfrak g$.
Suppose that $\d$ is a $\tau_{\a}$-derivation of $U(\mathfrak g)$.
Since $\mathfrak g$ is abelian, $\d(a_ia_j)=\d(a_ja_i)$ for all $1\<i,j\<n$, which means that
$\a_i\d(a_j)=\a_j\d(a_i)$, $1\<i,j\<n$. If $\d(\mathfrak g)\subseteq\mathfrak g$,
then the restriction $\d|_{\mathfrak g}$ can be regarded as a linear endomorphism
of $\mathfrak g$, i.e., $\d|_{\mathfrak g}\in{\rm End}(\mathfrak g)$.
By \leref{2.16}(a), there is a matrix $A=(\a_{ij})\in M_n(k)$ such that $\d|_{\mathfrak g}=\phi_A$.
In this case, $\a_i\d(a_j)=\a_j\d(a_i)$ $1\<i,j\<n$, if and only if
\begin{equation}\eqlabel{deri}
\a_i\a_{lj}=\a_j\a_{li},\ 1\<i,j, l\<n.
\end{equation}
Conversely, if a matrix $A=(\a_{ij})\in M_n(k)$ satisfies \equref{deri}, then
$\phi_A$ can be uniquely extended to a $\tau_{\a}$-derivation of $U(\mathfrak g)$, denoted by $\d_{\a, A}$.
In particular, if $\a=0$, $\d_{0, A}$ is a derivation of $U(\mathfrak g)$
for any $A\in M_n(k)$.

Now let $H$ be a generalized Hopf-Ore extension on $U(\mathfrak{g})$.
By \prref{2.3}, we may assume that $H=U(\mathfrak{g})(\chi_{\a},x,y,\d)$,
where $\a=(\a_1, \a_2, \cdots, \a_n)\in k^n$, $x, y\in\mathfrak{g}$
and $\d$ is a $\tau_{\a}$-derivation of $U(\mathfrak{g})$ such that \equref{e10} is satisfied for all $c\in\mathfrak{g}$.

Firstly, assume that $x=y=0$ or $kx=ky\neq 0$.
Then by \coref{1.12}(d), $H\cong U(\mathfrak{g})(\chi_{\a},0,0,\d)$,
a usual Hopf-Ore extension over $U(\mathfrak{g})$.
By \leref{2.1} and \equref{e10}, one knows that $\d(\mathfrak{g})\subseteq\mathfrak{g}$.
If $\a\neq 0$, then we may choose a matrix $P\in GL_n(k)$
such that $\a={\bf 1}P$. In this case, $\chi_{\a}=\chi_{\bf 1}\Phi_P$,
and it follows from \coref{1.12}(e) and \leref{2.7} that $H$ is isomorphic to
$U(\mathfrak{g})(\chi_{\bf 1},0,0,\Phi_P\d\Phi_{P^{-1}})$. Obviously,
$\Phi_P\d\Phi_{P^{-1}}(\mathfrak g)\subseteq\mathfrak g$, and hence $\Phi_P\d\Phi_{P^{-1}}=\d_{{\bf 1}, A}$
for some $A=(\a_{ij})\in M_n(k)$. Since $\d_{{\bf 1}, A}$ is a $\tau_{\bf 1}$-derivation, it follows from
\equref{deri} that $\a_{lj}=0$ for all $1\<l\<n$ and $2\<j\<n$.
Therefore, $\d_{{\bf 1},A}(a_1)=\sum_{j=1}^n\a_{j1}a_j$ and $\d_{{\bf 1},A}(a_i)=0$ for all $2\<i\<n$.
Let $b=-\d_{{\bf 1},A}(a_1)$. Then $\D(b)=b\ot 1+1\ot b$.
Define $\d'': U(\mathfrak g)\ra U(\mathfrak g)$ by $\d''(u)=\tau_{\bf 1}(u)b-bu$,
$u\in U(\mathfrak g)$. Then $\d''(a_1)=\tau_{\bf 1}(a_1)b-ba_1=(a_1+1)b-ba_1=b=-\d_{{\bf 1},A}(a_1)$
and $\d''(a_i)=\tau_{\bf 1}(a_i)b-ba_i=a_ib-ba_i=0$ for all $2\<i\<n$.
Hence $\d''=-\d_{{\bf 1},A}$, and so
$U(\mathfrak{g})(\chi_{\bf 1},0,0, \d_{{\bf 1},A})\cong U(\mathfrak{g})(\chi_{\bf 1},0,0, 0)$
by \prref{1.11}.
Let $A, B\in M_n(k)$. By \prref{1.11} and \leref{2.16}(b), one can check that
$U(\mathfrak{g})(\ep,0,0,\d_{0,A})\cong U(\mathfrak{g})(\ep,0,0,\d_{0,B})$
as generalized Hopf-Ore extensions if and only if $B=\l PAP^{-1}$ for some $\l\in k^{\times}$ and $P\in GL_n(k)$.

Now assume that $x, y\in \mathfrak{g}$ are linearly independent.
Then there is an invertible matrix $A$ in $GL_n(k)$ such that $\Phi_A(x)=a_1$ and $\Phi_A(y)=a_2$.
Since $\chi_{\a}\Phi_{A^{-1}}=\chi_{\a A^{-1}}$, we have generalized Hopf-Ore extension isomorphism
$U(\mathfrak{g})(\chi_{\a},x,y,\d)\cong U(\mathfrak{g})(\chi_{\a A^{-1}},a_1,a_2,\Phi_A\d\Phi_{A^{-1}})$
by \coref{1.12}(e). Thus, we may assume that $H=U(\mathfrak{g})(\chi_{\a},a_1,a_2,\d)$.
Then by \equref{e10}, we have
$$\D(\d(a_i))=\d(a_i)\ot 1+1\ot \d(a_i)-\a_i a_1\ot a_2, \ 1\<i\<n.$$
By \leref{2.2}(a), $\a_i=0$, $1\<i\<n$. Hence $\a=0$ and $\d(\mathfrak{g})\subseteq \mathfrak{g}$.
It follows that $\d=\d_{0,A}$ for some $A\in M_n(k)$.

\begin{lemma}\lelabel{iso}
Let $A=(\a_{ij}), B=(\b_{ij})\in M_n(k)$. Then $U(\mathfrak{g})(\ep,a_1,a_2,\d_{0,A})$ is isomorphic to
$U(\mathfrak{g})(\ep,a_1,a_2,\d_{0,B})$ as a generalized Hopf-Ore extension of $U(\mathfrak{g})$
if and only if there is a matrix $P=(p_{ij})\in GL_n(k)$ with
$p_{i1}=p_{i2}=0$ for all $2<i\<n$ such that $B=(p_{11}p_{22}-p_{12}p_{21})^{-1}PAP^{-1}$.
\end{lemma}

\begin{proof}
Suppose $U(\mathfrak{g})(\ep,a_1,a_2,\d_{0,A})\cong U(\mathfrak{g})(\ep,a_1,a_2,\d_{0,B})$
as generalized Hopf-Ore extensions of $U(\mathfrak{g})$. Then by \prref{1.11} and \leref{2.16},
there is an element $b\in U(\mathfrak g)$, a scale $\l\in k^{\times}$ and a matrix
$P=(p_{ij})\in GL_n(k)$ such that $\D(b)=b\ot 1+1\ot b+\l\Phi_P(a_1)\ot\Phi_P(a_2)-a_1\ot a_2$
and $\d_{0, B}=\l\Phi_P\d_{0,A}\Phi_{P^{-1}}$. Hence $B=\l PAP^{-1}$ and
$$\begin{array}{rl}
\D(b)=&b\ot 1+1\ot b+\l\sum_{i,j=1}^np_{i1}p_{j2}a_i\ot a_j-a_1\ot a_2\\
=&b\ot 1+1\ot b+\l\sum_{i=1}^np_{i1}p_{i2}a_i\ot a_i\\
&+(\l p_{11}p_{22}-1)a_1\ot a_2+\l p_{21}p_{12}a_2\ot a_1\\
&+\sum_{3\<i\<n}\l(p_{11}p_{i2}a_1\ot a_i+p_{i1}p_{12}a_i\ot a_1)\\
&+\sum_{2\<i<j\<n}\l(p_{i1}p_{j2}a_i\ot a_j+p_{j1}p_{i2}a_j\ot a_i).\\
\end{array}$$
Then by \leref{2.2}(a), we have $\l(p_{11}p_{22}-p_{12}p_{21})=1$,
$p_{11}p_{i2}=p_{i1}p_{12}$, $3\<i\<n$, and $p_{i1}p_{j2}=p_{j1}p_{i2}$,
$2\<i<j\<n$. It follows that $p_{i1}=p_{i2}=0$ for all $3\<i\<n$ and $\l=(p_{11}p_{22}-p_{12}p_{21})^{-1}$.

Conversely, suppose that there is a matrix $P=(p_{ij})\in GL_n(k)$ with
$p_{i1}=p_{i2}=0$ for all $2<i\<n$ such that $B=(p_{11}p_{22}-p_{12}p_{21})^{-1}PAP^{-1}$.
Then by det$(P)\neq 0$ and $p_{i1}=p_{i2}=0$ for all $2<i\<n$, one gets $p_{11}p_{22}-p_{12}p_{21}\neq 0$.
Let $\l=(p_{11}p_{22}-p_{12}p_{21})^{-1}$ and $b=\frac{1}{2}\l\sum_{i=1}^2 p_{i1}p_{i2}a_i^2+\l p_{12}p_{21}a_1a_2$.
Then $\l p_{11}p_{22}-1=\l p_{12}p_{21}$, $\d_{0, B}=\l\Phi_P\d_{0,A}\Phi_{P^{-1}}$ and
$$\begin{array}{rl}
\D(b)=&\frac{1}{2}\l\sum_{i=1}^2p_{i1}p_{i2}(a_i\ot 1+1\ot a_i)^2\\
&+\l p_{12}p_{21}(a_1\ot 1+1\ot a_1)(a_2\ot 1+1\ot a_2)\\
=&\frac{1}{2}\l\sum_{i=1}^2p_{i1}p_{i2}a_i^2\ot 1+1\ot \frac{1}{2}\l\sum_{i=1}^2p_{i1}p_{i2}a_i^2+\l\sum_{i=1}^2p_{i1}p_{i2}a_i\ot a_i\\
&+\l p_{12}p_{21}a_1a_2\ot 1+\l p_{12}p_{21}1\ot a_1a_2+\l p_{12}p_{21}a_1\ot a_2+\l p_{12}p_{21}a_2\ot a_1\\
=&b\ot 1+1\ot b+\l\sum_{i=1}^2p_{i1}p_{i2}a_i\ot a_i\\
&+(\l p_{11}p_{22}-1)a_1\ot a_2+\l p_{21}p_{12}a_2\ot a_1\\
=&b\ot 1+1\ot b+\l\sum_{i,j=1}^2p_{i1}p_{j2}a_i\ot a_j-a_1\ot a_2\\
=&b\ot 1+1\ot b+\l\Phi_P(a_1)\ot\Phi_P(a_2)-a_1\ot a_2.\\
\end{array}$$
It follows from \prref{1.11} that
$U(\mathfrak{g})(\ep,a_1,a_2,\d_{0,A})\cong U(\mathfrak{g})(\ep,a_1,a_2,\d_{0,B})$
as generalized Hopf-Ore extensions of $U(\mathfrak{g})$.
\end{proof}

Define a relation $\sim$ on $M_n(k)$ as follows: for $A, B\in M_n(k)$,
$A\sim B$ if and only if there is a scalar $\l\in k^{\ti}$ and a matrix $P\in GL_n(k)$
such that $B=\l PAP^{-1}$. Then $\sim$ is an equivalence relation on $M_n(k)$.
Let $M_n(k)/\sim$ be the corresponding set of equivalence classes.
Similarly, define a relation $\sim'$ on $M_n(k)$ as follows:
for $A, B\in M_n(k)$,
$A\sim' B$ if and only if there is a matrix $P=(p_{ij})\in GL_n(k)$ with $p_{i1}=p_{i2}=0$
for all $3\<i\<n$ such that $B=(p_{11}p_{22}-p_{12}p_{21})^{-1}PAP^{-1}$.
Then $\sim'$ is also an equivalence relation on $M_n(k)$.
Let $M_n(k)/\sim'$ be the corresponding set of  equivalence classes.

\begin{proposition}\prlabel{2.16}
Each generalized Hopf-Ore extension of $U(\mathfrak{g})$ is isomorphic to one of the followings:

{\rm(a)} $U(\mathfrak{g})(\ep,0,0,\d_{0,A})$, where $A\in M_n(k)/\sim$;\\
{\rm(b)} $U(\mathfrak{g})(\chi_{\bf 1},0,0,0)$;\\
{\rm(c)} $U(\mathfrak{g})(\ep,a_1,a_2,\d_{0,B})$, where $B\in M_n(k)/\sim'$.

Moreover, as generalized Hopf-Ore extensions of $U(\mathfrak{g})$,
$U(\mathfrak{g})(\ep,0,0,\d_{0,A})$, $U(\mathfrak{g})(\chi_{\bf 1},0,0,0)$
and $U(\mathfrak{g})(\ep,a_1,a_2,\d_{0,B})$ are pairwise non-isomorphic.
\end{proposition}
\begin{proof}
The first claim follows from the discussion above and \leref{iso}, the second one follows from \prref{1.11}.
\end{proof}

When $n=2$, $\mathfrak g$ is exactly the Lie algebra $\mathfrak{g}_1$ given in Subsection 2.3.
Thus, let $n=2$ in \prref{2.16}, one gets the classifications of the generalized Hopf-Ore extensions of
$U(\mathfrak{g}_1)$ with char$(k)=0$. In this case, two matrices $A, B\in M_2(k)$ satisfy
$A\sim' B$ if and only if there is a matrix $P\in GL_2(k)$ such that $B={\rm det}(P)^{-1}PAP^{-1}$.

\begin{remark}
	Assume that $k$ is an algebraically closed field of characteristic zero.
	In \cite{Zh}, author constructed two kinds of connected Hopf algebras $A(\l_1,\l_2,\a)$ and $B(\l)$ and
	proved that every connected Hopf algebra of GK-dimension three is isomorphic to one of the following: the enveloping
	algebra $U(\mathfrak{g})$ for a $3$-dimensional Lie algebra $\mathfrak{g}$,
	the Hopf algebras $A(0,0,0)$, $A(0,0,1)$, $A(1,1,1)$, $A(1,\lambda,0)$  and $B(\l)$, $\lambda\in k$.
	Clearly, one can check that these Hopf algebras are the generalized Hopf-Ore extensions of the enveloping
	algebras of some $2$-dimensional Lie algebras.
\end{remark}

\section*{ACKNOWLEDGMENTS}

This work is supported by the National Natural Science Foundation of China (Grant No. 11571298).


\begin{thebibliography}{99}
\bibitem{BDG1}
M. Beattie, S. D\v{a}sc\v{a}lescu, L. Gr\"{u}nenfelder, On the number of types of finite dimensional Hopf algebras, Invent. Math. 136(1) (1999)1-7.
\bibitem{BDG2}
M. Beattie, S. D\v{a}sc\v{a}lescu, L. Gr\"{u}nenfelder, Constructing pointed Hopf algebras by Ore extensions, J. of Algebra 225(2)(2000) 743-770.
\bibitem{BOZZ}
K.A. Brown, S. O'Hagan, J.J. Zhang and G. Zhuang, Connected Hopf algebras and
iterated Ore extensions, J. of Pure and Applied Algebra 219(6) (2015) 2405-2433.
\bibitem{Ci}
C. Cibils, Half-quantum groups at roots of unity, path algebras, and representation type,  Int. Math. Res. Notices (1997) 1997 (12): 541-553.
\bibitem{Mc}
J.C. McConnell and J.C.Robson, Noncommutative Noetherian rings, Wiley-Interscience, New York. 1987.
\bibitem{Mo}
S. Montgomery, Hopf Algebras and Their Actions on Rings, CBMS Reg. Conf. Ser. Math.
82, Amer. Math. Soc., Providence, RI, 1993.
\bibitem{Pa}
A. N. Panov, Ore extensions of Hopf algebras, Mathematical Notes 74(3) (2003) 401-410.
\bibitem{WYC}
Z. Wang, L. You, H. X. Chen, Representations of Hopf-Ore extensions of group algebras
and pointed Hopf algebras of rank one, Algebras and Representation Theory 18(3) (2015) 801-830.
\bibitem{Zh}
G. Zhuang, Properties of connected Hopf algebras of finite Gelfand-Kirillov dimension, J.
London Math. Soc. 87 (2)(2013) 877 - 898.

\end{thebibliography}
\end{document}